\documentclass{article}
\usepackage{amsmath,amssymb,amsfonts,amsthm}
\usepackage{xfrac}

\newcommand{\R}{\mathbb{R}}
\newcommand{\Z}{\mathbb{Z}}
\newcommand{\N}{\mathbb{N}}
\newcommand{\Ph}{\R^N}
\newcommand{\E}{\ensuremath{\mathbb{E}}}
\newcommand{\D}{\R^d}
\newcommand{\DTwo}{\R^2}
\newcommand{\DT}{\D\times\R^+}

\newcommand{\TD}{\R^+}
\newcommand{\DTT}{\D\times[0,T)}
\newcommand{\DP}{\R^{N\times d}}
\newcommand{\Meas}{\ensuremath{\mathcal{M}}}
\newcommand{\Prm}{\ensuremath{\text{Prob}}}
\newcommand{\dd}{\ensuremath{\:d}}
\newcommand{\ddt}[1]{\ensuremath{\frac{d#1}{dt}}}
\newcommand{\Dx}{\ensuremath{h}}
\newcommand{\Dy}{\ensuremath{\Delta y}}
\newcommand{\Dt}{\ensuremath{\Delta t}}
\newcommand{\Ypair}[2]{\langle #1, #2\rangle}
\newcommand{\NumericalEvolution}[1]{\ensuremath{\mathcal{S}^{#1}}}
\newcommand{\Y}{\ensuremath{\textbf{Y}}}
\newcommand{\Prb}{\ensuremath{\mathbb{P}}}

\newcommand{\weakto}{\ensuremath{\rightharpoonup}}
\newcommand{\Mesh}{\ensuremath{\mathcal{M}}}
\newcommand{\algref}[1]{Algorithm~\ref{#1}}
\newcommand{\hf}{\sfrac{1}{2}}
\newcommand{\Borel}{\ensuremath{\mathcal{B}}}
\newcommand{\FKMT}[3]{\ensuremath{E^{#1}_{#3}(#2)}}
\newcommand{\divx}[1]{\ensuremath{\nabla_x \cdot #1}}
\newcommand{\gradx}[1]{\ensuremath{\nabla_x #1}}
\newcommand{\weakdiffpg}[2]{\ensuremath{\Ypair{\psi}{\Ypair{#1-#2}{g}}}}

\newcommand{\WeakErrorDiffMC}[3]{\frac{\sqrt{\Var(\Ypair{\psi}{g(#1)-g(#2)}})}{#3}}
\newcommand{\WeakErrorMC}[2]{\frac{\sqrt{\Var(\Ypair{\psi}{g(#1)})}}{#2}}
\newcommand{\StochasticLTwoNorm}[1]{\ensuremath{\left\|#1\right\|_{L^2(\Omega)}}}

\newcommand{\WorkFVM}[2]{\text{Work}_{\text{FVM}}(#1, #2)}
\newcommand{\WorkFVMCFL}[1]{\text{Work}_{\text{FVM}}(#1)}
\newcommand{\WorkMLMC}[1]{\text{Work}_{\text{MLMC}}(#1)}
\newcommand{\WorkMC}[2]{\text{Work}_{\text{MC}}(#1,#2)}
\newcommand{\MLMC}[3]{E_{\text{MLMC}, #1}^{#2}(#3)}
\newcommand{\bigO}{\mathcal{O}}
\newcommand{\pdpd}[2]{\ensuremath{\frac{\partial #1}{\partial #2}}}
\usepackage{subcaption}

\DeclareMathOperator{\id}{id}
\DeclareMathOperator{\Var}{Var}
\DeclareMathOperator{\Law}{Law}
\DeclareMathOperator{\Lip}{Lip}
\DeclareMathOperator{\SpeedUp}{SpeedUp}
\usepackage{authblk}

\usepackage{todonotes}
\newtheorem{theorem}{Theorem}
\theoremstyle{definition}
\newtheorem{definition}{definition}
\newtheorem{corollary}{Corollary}
\newtheorem{lemma}{Lemma}
\newtheorem{remark}{Remark}
\newtheorem{Algorithm}{Algorithm}
\usepackage{cleveref}
\usepackage{url}

\title{Multilevel Monte-Carlo for measure valued solutions \thanks{This work was funded by the ERC project SPARCCLE 306279}}

\author{
	Kjetil Olsen Lye\thanks{Seminar for applied mathematics, ETH Zurich,
		(kjetil.lye@sam.math.ethz.ch).}}

\begin{document}
	
	\maketitle
	
\begin{abstract}
	We propose a Multilevel Monte-Carlo (MLMC) method for computing entropy measure valued solutions of hyperbolic conservation laws. Sharp bounds for the narrow convergence of MLMC for the entropy measure valued solutions are proposed. An optimal work-vs-error bound for the MLMC method is derived assuming only an abstract decay criterion on the variance. Finally, we display numerical experiments of cases where MLMC is, and is not, efficient when compared to Monte-Carlo.
\end{abstract}


A class of equations that is of great interest for both physics and engineering, is the class of hyperbolic conservation laws, of the form
\begin{equation}
\label{eq:hyperbolic_conservation_law}
\begin{array}{rl}
u_t+\divx f(u)=0\\
u(x,0) = u_0(x).
\end{array}
\end{equation}
Here $u:\DT\to \Ph$ is the unknown and $f:\Ph\to \DP$ is the flux function. We will concern ourselves with the hyperbolic case, in other words when $\partial_u (f\cdot n)$ has real eigenvalues for each $|n|=1$.

Examples of this class include the shallow water equations, the compressible Euler equations for gas dynamics and the magneto hydrodynamics equations for plasmas.  For a comprehensible introduction to hyperbolic conservation laws, consult~\cite{Dafermos}.
 
\section{Introduction}
It is well known that solutions of \eqref{eq:hyperbolic_conservation_law} can develop discontinuities in finite time, and one therefore needs to consider a weak formulation. 
\begin{definition}
	We say that $u\in L^\infty(\DT, \Ph)$ is a weak solution of \eqref{eq:hyperbolic_conservation_law} if
	\[\int_{\TD}\int_{\D} u(x,t)\phi_t(x,t)+\gradx{\phi}(x,t)\cdot f(u(x,t))\phi_x(x,t)\dd x \dd y+\int_{\TD}u_0(x)\phi(x,0)\dd x,\]
	holds for every $\phi \in C_c^1(\DT, \R)$.
\end{definition}
Weak solutions of \eqref{eq:hyperbolic_conservation_law} are in general not unique, therefore one seeks the physical relevant solutions, in terms of  entropy conditions.
\begin{definition}
A pair $(\eta, q)$ with $\eta:\Ph\to \R$ and $q:\Ph\to \D$ is called an entropy pair if $\eta$ is convex and $q'=\eta'\cdot f'$.
\end{definition}
\begin{definition}
A weak solution $u$ of \eqref{eq:hyperbolic_conservation_law} is an entropy solution if the entropy inequality
\[\eta(u)_t+\gradx{q(u)}\leq 0\]
holds for all entropy pairs $(\eta, q)$, that is if 
\[\int_{\TD}\int_{\D}\eta(u(x,t))\phi_t(x,t)+\gradx{\phi(x,t)}\cdot q(u(x,t))\dd x\dd t\geq 0,\]
for all $0\leq \phi\in  C_c^1((0,\infty), \R)$.
\end{definition}
It is by now well-known that entropy solutions of scalar conservation laws ($N=1$) are unique and well-posed \cite[6.2.3]{Dafermos}. 
On the other hand, no global well-posedness result is available for a generic system of conservation laws in several space dimensions. In fact, one can construct multiple weak entropy solutions with the same initial data~\cite{delellis_nonunique1,delellis_nonunique2}.

In addition, recent theoretical~\cite{glimm1,glimm2} and numerical evidence~\cite{fkmt} indicate that a more appropriate notion of solution for conservation laws is the notion of a measure valued solution, as introduced by DiPerna~\cite{diperna}.

\subsection{Numerical approximation of conservation laws}

By now, there is a large set of numerical methods for approximating solutions of \eqref{eq:hyperbolic_conservation_law}. Popular choices include the finite volume (difference) based methods~\cite{leveque_green}, using TVD~\cite{TVD}, ENO~\cite{ENO} and WENO~\cite{WENO} reconstruction. Another popular method is the discontinuous Galerkin method~\cite{DG}.

Convergence results are available for one-dimensional scalar equations~\cite{conv_monotone_1d} and multi-dimensional scalar equations~\cite{conv_monotune_multid} assuming the numerical scheme is monotone. Convergence results for abitrarily high order schemes are also available~\cite{tecno} in the scalar case. There are also results available for the discontinuous Galerkin method~\cite{hiltebrand_stdg} for scalar conservation laws.

However, for systems of conservation laws in several space dimensions, \emph{no} convergence result is known. In fact, recent numerical studies show that there are initial data for the compressible Euler equations, where numerical schemes may fail to converge~\cite{fkmt}.

\subsection{Uncertainty quantification}

The numerical methods outlined in the previous section all rely on measuring the initial data $u_0$ exactly. However, in real world scenarios, accurate measurements of the initial data may not be available, and it is common to model the initial data $u_0$ and the corresponding solution $u$, as a random field. This approach is commonly known as uncertainty quantification (UQ). 

There is a large set of methods using the approach of random fields for hyperbolic conservation laws. The Monte-Carlo method has been shown to be robust and reliable for a very wide variety of problems in UQ for conservation laws. However, it suffers from the high runtime cost. The multilevel Monte-Carlo (MLMC) method, first introduced by Heinrich \cite{Heinrich2001} for parametric integration, and later by Giles \cite{giles} for stochastic equations, has been shown to have a considerable speed-up in the case of scalar conservation laws \cite{mlmc_hyperbolic}. 

There are also other approaches that exploit further regularity of the solutions, including stochastic finite volume methods~\cite{sfvm1,sfvm2} and stochastic collocation methods~\cite{Tokareva2013}.  

While the approach through random fields has been very successful for scalar hyperbolic conservation laws, its theoretical foundation relies on \eqref{eq:hyperbolic_conservation_law} being well-posed to pick the solution $u(\omega;\cdot)$ for almost all $\omega$ as in~\cite{mlmc_hyperbolic}. However, as has been indicated by both theory~\cite{delellis_nonunique1} and numerical experiments~\cite{fkmt}, non-linear systems of conservation laws need not be well-posed, and hence may not have a unique solution. The framework for UQ, as found in~\cite{mlmc_hyperbolic}, is therefore not applicable for systems of conservation laws, and the results obtained for scalar conservation laws do not apply for system of conservation laws.

\subsection{Measure valued solutions}
A weaker notion of solutions is the notion of a measure valued solution of conservation laws, as first introduced by DiPerna~\cite{diperna}. One seeks a measure valued function $(x,t)\mapsto \nu_{x,t}\in \Prm(\Ph)$ satisfying \eqref{eq:hyperbolic_conservation_law} in the sense of measures.

One can embed uncertainty quantifications within the framework of measure valued solutions, and indeed measure valued solutions have the clear advantage that they do not require a notion of a weak path solution of the underlying deterministic conservation law. In essence, the definition of a measure valued solution does not rely on a well-posed underlying deterministic equation, whereas the definition of a random field solution does. 
s
Fjordholm et al~\cite{fkmt} developed and tested a numerical algorithm, the so-called FKMT algorithm, for computing measure valued solutions of conservation laws. Through numerical experiments, the algorithm exhibited convergence and stability in the space of Young measures. The FKMT algorithm uses a Monte-Carlo sampling procedure which has a sampling error that scales as $M^{-\hf}$, where $M$ is the number of samples. The sampling error gives the Monte-Carlo algorithm a high computational cost.

\subsection{Scope of paper}
The goal of this paper is to develop a multilevel Monte-Carlo algorithm for computing measure valued solutions of conservation laws. The key ingredient in MLMC method is the error bound, which in turn can be used to determine the optimal number of samples. 

We obtain an error estimate in the narrow topology for the MLMC method in \Cref{thm:weak_mlmc}, involving an abstract variance decay rate. We use this abstract decay rate to find the asymptotically optimal number of samples per level in \Cref{thm:number_of_samples}. The key insight from these two theorems combined, will be that we need a decay on the variance in order for MLMC to get a speedup compared to ordinary singlelevel Monte-Carlo.

In the case of a scalar conservation law, we can produce a rate for the variance decay, and the numerical experiments confirm this. We will see that in this case, the MLMC method outperforms the singlelevel Monte-Carlo method by two order.

For system of conservation laws, no known general variance reduction is known. We will therefore rely on numerical experiments to determine the variance reduction in each case. We perform two different numerical experiments. In the first experiment, the shockvortex interaction, we get variance reduction and the MLMC algorithm produces a speedup against singlelevel Monte-Carlo. However, for the second example, the Kelvin-Helmholtz initial data, there is no variance reduction. We furthermore observe that in the case of the Kelvin-Helmholtz initial data, MLMC produces no speed-up compared to singlelevel Monte-Carlo. 

\section{Measure valued solutions}
\label{sec:mvs}
For a measurable space $(X, \Sigma)$, we let $\Meas(X)$ denote the sets of (signed) measures on $(X,\Sigma)$ and  $\Prm(X)$ denote the set of probability measures on $(X, \Sigma)$. We will often concern ourselves with the case of a domain $X=D\subset \R^n$ and $\Sigma$ being the Borel $\sigma$-algebra on $D$. 

A map  $\nu: D\to \Prm(\R^N)$, is called a \emph{Young measure} from $D$ to $\R^N$ if for every $g\in C_0(\R^N)$, the map
\[D\ni z\mapsto\langle \nu_z, g\rangle :=\int_{\R^N}g(\xi)\dd \nu_z(\xi)\]
is measurable for almost all $z\in D$. We let $\Y(D,\R^N)$ denote the set of Young measures from $D$ to $\R^N$.

We interpret $\Ypair{\nu}{g}$ as the \emph{expectation} of $g$ with respect to the probability measure $\nu$. We can obtain all known one-point statistics on this form. The mean is given as
\[\E(g) = \Ypair{\nu}{g},\]
and the variance  can in a similar manner be given as
\[\Var(g) = \Ypair{\nu}{g\otimes g}-\Ypair{\nu}{g}^2.\]
We say that a sequence $\{\nu_n\}_n$ in $\Y(D,\R^N)$ converges narrowly to $\nu\in \Y(D,\R^N)$, if for all $\psi\in L^1(\R^N)$ and $g\in C_b(D)$ we have
\[\lim_{n\to \infty}\Ypair{\psi}{\Ypair{\nu_n}{g}} = \Ypair{\psi}{\Ypair{\nu}{g}}:=\int_{D}\psi(z)\Ypair{\nu_z}{g}\dd z.\]

It can be readily seen that narrow convergence implies convergence of statistical quantities, as given above.

For a probability space $(\Omega, \Sigma, \Prb)$, it can be shown \cite{fkmt} that every random variable $u: \Omega\times D\to \Ph$ gives rise to a Young measure $\nu\in\Y(D,\Ph)$ through
\[\nu_x=\Law(u(\cdot, x)).\]
Here,
\[\Law(u(\cdot, x))(A):=\Prb(u(\cdot, x)^{-1}(A))\qquad A\subset \Borel(\R^N),\]
and $\Borel(\R^N)$ denotes the Borel sets of $\R^N$.
Conversely, any Young measure can be represented as the law of a random variable, as the following theorem makes precise.

\begin{theorem}[\cite{fkmt}]
	\label{thm:young_to_randvar}
	Let $\nu\in\Y(D,\Ph)$, then there exists a probability space $(\Omega,\Sigma,\Prb)$ and a function $u:\Omega\times D\to \Ph$ such that
	\begin{equation}
	\Law(u(\cdot, x))=\nu_x\qquad \text{for all } x\in D.
	\end{equation}
\end{theorem}

\subsection{Measure valued solutions}
We recast \eqref{eq:hyperbolic_conservation_law} to a measure valued equation given as
\begin{equation}
\label{eq:measure_valued_conservation_law}
\begin{array}{rl}
\Ypair{\nu}{\id}_t+\divx{\Ypair{\nu}{f}} =0\\
\nu_{x,0} = \sigma_x,
\end{array}
\end{equation}
interpreted in the sense of distributions.
\begin{definition}
A Young measure $\nu\in \Y(\D, \Ph)$ is said to be a measure valued solution (MVS) of \eqref{eq:measure_valued_conservation_law} if 
\begin{equation}
\label{eq:weak_measure_valued_conservation_law}
\int_{\DT}\partial_t \phi\Ypair{\nu}{\id}+\nabla_x \phi\cdot \Ypair{\nu}{f}\dd x \dd t+\int_{\D}\phi(x,0)\Ypair{\sigma}{\id}\dd x=0
\end{equation}
for all test functions $\phi\in C_c^1(\DT, \Ph)$.
\end{definition}
We furthermore define the notion of an \emph{entropy measure valued solution} in an analogous manner.

\begin{definition}
	We say that a measure valued  solution $\nu$ of \eqref{eq:hyperbolic_conservation_law} is an entropy measure valued solution(EMVS) if
	for all entropy pairs $(\eta, q)$, 
	\[\int_{\TD}\int_{\D}\Ypair{\nu_{x,t}}{\eta}\phi_t(x,t)+\gradx{\phi(x,t)}\cdot \Ypair{\nu_{x,t}}{q}\dd x\dd t\geq 0,\]
	for all $0\leq \phi\in  C_c^1(\DT, \R)$.
\end{definition}

\subsection{Non-atomic initial data and measure valued solutions as UQ}
\label{subsec:nonatomic_uq}
As was described in the previous section, in real applications, measurement errors in the initial data $u_0$ are unavoidable. Therefore, it is common to measure the initial data as a random field $u_0:\Omega\times \D\to \Ph$ and solve \eqref{eq:hyperbolic_conservation_law} for each $u(\omega, \cdot)$, and estimates the statistics (expectation, variance and so on). However, as explained above, we know that
\[\sigma_x:=\Law(u_0(\cdot; x))\qquad x \in \D\]
will be a Young measure. By setting the initial data to $\sigma$, one can use the EMVS to do uncertainty quantifications, ie. measure the mean, variance and other one point statistics.
 
\subsection{Approximation of conservation laws}
This section briefly describes the conventional way of discretizing conservation laws through finite volume and finite difference methods. For a complete review, one can consult~\cite{leveque_green}.

We discretize the computational spatial domain as a collection of cells 
\[\{(x^1_{i_1-\hf}, x^1_{i_1+\hf})\times\cdots\times(x^d_{i^d-\hf}, x^d_{i^d+\hf})\}_{(i^1,\ldots,i^d)}\subset \D,\]
with corresponding cell midpoints
\[x_{i^1,\ldots, i^d}:=\left(\frac{x^1_{i_1+\hf}+x^1_{i_1-\hf}}{2},\ldots,\frac{x^d_{i_d+\hf}+x^d_{i_d-\hf}}{2}\right).\] 
For simplicity, we assume our mesh is equidistant, and set 
\[\Dx := |x^1_{i^1+\hf} - x^1_{i^1-\hf}|.\]
We will describe the semi-discrete case. For each cell $(i^1,\ldots, i^d)$, we let $u^{\Dx}_{i^1,\ldots, i^d}(t)$ denote the averaged value in the cell at time $t\geq 0$.

We use the following semi-discrete formulation
\begin{equation}
\label{eq:semi_d}\begin{aligned}\ddt{}u^{\Dx}_{i^1,\ldots,i^d}(t)+\sum_{k=1}^d\frac{1}{\Dx}\left(F^{k,\Dx}_{i^d,\ldots,i^k+\hf,\ldots, i^d}(t)-F^{k,\Dx}_{i^d,\ldots,i^k-\hf,\ldots, i^d}(t)\right)&=0\\
u^{\Dx}_{i^1,\ldots,i^d}(0)&=u_0(x_{i^1,\ldots,i^d}).
\end{aligned}
\end{equation}

Where we have used a \emph{numerical flux function} $F^k_{i^1,\ldots,i^d}$. In this paper, the numerical flux function will always have a finite stencil width, meaning $F^{k,\Dx}_{i^1,\ldots,i^k,\ldots,i^d}(t)$ will only depend on $u^{\Dx}_{i^1,\ldots,j^k,\ldots,i^k}(t)$ for $j^k=i^k-p+1,\ldots,i^k+p$.

We furthermore assume the numerical flux function is consistent with $f$ and locally Lipschitz continuous, which amounts to requiring that for every compact $K\subset \D$, there exists a constant $C>0$ such that for $k=1,\ldots,d$, it holds that
\begin{equation}
|F^{k,\Dx}_{i^1,\ldots,i^d}(t)-f(u^{\Dx}_{i^1,\ldots,i^d})|\leq C\sum_{j^k=i^d-p+1}^{i^k+p}|u^{\Dx}_{i^1,\ldots,i^d}(t)-u^{\Dx}_{i^1,\ldots, j^k,\ldots,i^d}(t)|,
\end{equation}
whenever $\{u^{\Dx}_{i^1,\ldots, i^k-p+1,\ldots,i^d}(t),\ldots, u^{\Dx}_{i^1,\ldots, i^k+p,\ldots, i^d}(t)\}\subset K$.

We let 
$\NumericalEvolution{\Dx}:L^\infty(\D,\Ph)\to L^\infty(\DT,\Ph)$
be the discrete numerical evolution operator corresponding to \eqref{eq:semi_d}.

The current form of \eqref{eq:semi_d} is continuous in time, and one needs to employ a time stepping method to discrete the ODE in time, usually through some strong stability preserving Runge-Kutta method.

\subsection{Approximation of measure valued solutions}

In this section we repeat what is known for the approximation of measure valued solutions. 

Let $\sigma\in\Y(\D, \Ph)$ be the initial data, and choose $u_0$ according to \Cref{thm:young_to_randvar} such that the law of $u_0(\cdot, x)$ is $\sigma_x$. Introduce $u^{\Dx}:\Omega\times \DT\to \Ph$ by 
\[u^{\Dx}(\omega, x, t):=\NumericalEvolution{\Dx}(u_0(\omega,\cdot))(x,t).\]
We furthermore set
\[\nu^{\Dx}_{x,t} :=\Law(u^{\Dx}(\cdot, x, t).\]
In 2D, we have the following result \cite{fkmt}.

\begin{theorem}
	\label{thm:emvs_spatial_convergence}
	Assume the numerical scheme of $\NumericalEvolution{}$ satisfies the following requirements:
	\begin{enumerate}
		\item \textbf{Uniform boundednesss:}
		\begin{equation}
		\|u^{\Dx}(\omega)\|_{L^\infty(\DT)}\leq C\qquad\text{for all } \omega\in\Omega.
		\end{equation}
		\item \textbf{Weak BV:} There exists $1\leq r<\infty$ such that
		\begin{equation}
		\lim_{\Dx\to 0}\int_0^T\sum_{i,j}\left(|u^{\Dx}_{i+1,j}(\omega, t)-u^{\Dx}_{i,j}(\omega, t)|^r+|u^{\Dx}_{i,j+1}(\omega, t)-u^{\Dx}_{i,j}(\omega, t)|^r\right)\Dx^2 \dd t=0.
		\end{equation}
		\item \textbf{Entropy consistency.} The numerical scheme is entropy stable with respect to an entropy pair $(\eta,q)$, in the sense that there exists a Lipschitz numerical entropy flux $(Q^x_{i+1/2}(t),Q^y_{i, j+1/2}(t))$ consistent with the entropy flux $q$ such that the computed solution obeys the discrete entropy inequality
		\begin{equation}
		\eta(u^{\Dx})_t + \frac{1}{\Dx}\left(Q^{x,\Dx}_{i+1/2,j}-Q^{x,\Dx}_{i
			-1/2,j}\right)+\frac{1}{\Dy}\left(Q^{y,\Dx}_{i,j+1/2}-Q^{y,\Dx}_{i,j+1/2}\right)\leq 0,
		\end{equation}
		for all  $t>0$, $i,j\in\Z$, $\omega\in\Omega$.
		\item \textbf{Consistency with initial data.} If $\sigma^{\Dx}_x$ is the law of $u^{\Dx}(\cdot, x, 0)$, then 
		\begin{equation}
		\lim_{\Dx\to 0}\int_{\DTwo}\psi(x)\Ypair{\sigma^{\Dx}_x}{\id}\dd x=0\qquad \text{for all } \psi\in C_c^1(\DTwo)
		\end{equation}
		and
		\begin{equation}
		\limsup_{\Dx\to 0}\int_{\DTwo}\psi(x)\Ypair{\sigma^{\Dx}_x}{\eta}\dd x\leq 0\qquad \text{for all } 0\leq \psi\in C_c^1(\DTwo).
		\end{equation}
		
	\end{enumerate}
	Then up to a subsequence $\nu^{\Dx}$ converges to an entropy measure valued solution of \eqref{eq:hyperbolic_conservation_law} with initial data $\sigma$.
\end{theorem}
\begin{remark}
	The above theorem can be generalized to arbitrary space dimension, see \cite{fkmt}.
\end{remark}
\begin{remark}
	The above theorem tells us that the spatial discretization converges However, we are still left with the question of approximating the stochastic component. In other words, if we simulate for \emph{all} $\omega\in\Omega$, we will get a good approximation of $\nu$. Since $\Omega$ is in general (uncoutable) infinite, this is not a fruitful approach, and we refer to the next section for a solution.
\end{remark}
	
\subsection{The FKMT algorithm for computing measure valued solutions}
Fjordholm et al~\cite{fkmt} constructed a numerical Monte-Carlo based algorithm, the so-called FKMT algorithm, to compute measure valued solutions of \eqref{eq:measure_valued_conservation_law}. We describe the algorithm here for completeness and to establish notation.

\begin{Algorithm}
\label{alg:fkmt}
Let $\sigma \in\Y(\D, \Ph)$ be the initial data, and choose a probability space $(\Omega, \Sigma, \Prb)$ together with $u_0:\Omega\times \D\to\Ph$ such that $\sigma_x=\Law{u_0(\cdot, x)}$.
\begin{enumerate}
\item Draw $M$ independent samples $\{u_0^k\}_k$ of $u_0$.
\item Evolve the samples
\[u^{\Dx}_k = \NumericalEvolution{\Dx}(u_0^k)\]
\item Estimate the measure
\[\FKMT{M}{u_0}{\Dx}:=\frac{1}{M}\sum_{k=1}^M \delta_{u^{\Dx}_k}\]
\end{enumerate}
\end{Algorithm}

In~\cite{fkmt}, an error bound for the FKMT algorithm was obtained. Furthermore, if one follows the proof of~\cite[Theorem 4.9]{fkmt}, one can get a sharp bound on the stochastic error. We repeat the proof here for completeness. First we need a technical lemma showing the precise bound of the Monte-Carlo error.

\begin{lemma}
\label{lem:mc}
Let $(\Omega, \Sigma, \Prb)$ be a probability space, $M\in\N$  and $G\in L^2(\Omega)$, $\{G_k\}_{k=1}^M\subset L^2(\Omega)$ be independent identically distributed random variables.
Then
\[\StochasticLTwoNorm{\E(G)-\frac{1}{M}\sum_{k=1}^M G_k}^2=\frac{1}{M}\Var(G)\].
\end{lemma}
\begin{proof}
	We have
	\begin{multline*}
	\left(\E(G)-\frac{1}{M}\sum_{k=1}^M G_k(\omega)\right)^2=\frac{1}{M^2}\sum_{k=1}^M\Big(\E(G)- G_k(\omega)\Big)^2 \\
	+\frac{1}{M^2}\sum_{k=1}^M\sum_{k'\neq k}\Big(\E(G)- G_k(\omega)\Big)\Big(\E(G)- G_{k'}(\omega)\Big).
	\end{multline*}
	Since $G_k$ are independent, we have for $k\neq k'$
	\begin{align*}\E\big(\left(\E(G)- G_k(\omega)\right)\left(\E(G)- G_{k'}(\omega)\right)\big)&=\E\Big(\E(G)- G_k(\omega)\Big)\E\Big(\E(G)- G_{k'}(\omega)\Big)\\
	&=\left(\E(G)- \E(G_k(\omega))\right)\left(\E(G)- \E(G_{k'}(\omega))\right)\\
	&=0.
	\end{align*}
	
	We end up with
	\begin{align*}
	\StochasticLTwoNorm{\E(G)-\frac{1}{M}\sum_{k=1}^M G_k}^2&=\E(\Big(\E(G)-\frac{1}{M}\sum_{k=1}^M G_k\Big)^2)\\
	&=\frac{1}{M^2}\sum_{k=1}^M\E(\Big(\E(G)- G_k(\omega)\Big)^2)\\
	&=\frac{1}{M^2}\sum_{k=1}^M\left(\E(G^2) - \E(G)^2\right)\\
	&=\frac{1}{M}\Var(G).
	\end{align*}
\end{proof}
\begin{theorem}
	\label{thm:fkmt_mc}
If $\NumericalEvolution{\Dx}$ obeys the requirements of \Cref{thm:emvs_spatial_convergence}, then \algref{alg:fkmt} converges, that is up to subsequences, we have
\[\FKMT{M}{u_0}{\Dx}\weakto \nu,\]
where $\nu$ is a entropy measure valued solution of \eqref{eq:hyperbolic_conservation_law}.
 
Concretely, for every $\psi\in L^1(\DT)$ and $g\in C_b(\Ph)$, we have

\begin{equation}
\label{eq:fkmt_error_sharp}
\StochasticLTwoNorm{\weakdiffpg{\FKMT{M}{u}{\Dx}}{\nu^{\Dx}}}=\frac{1}{M^{1/2}}\sqrt{\Var(\Ypair{\psi}{g(u^{\Dx})})}.
\end{equation}
Up to subsequences, $\nu^{\Dx}\weakto \nu$.
\end{theorem}
\begin{proof}
By \Cref{thm:emvs_spatial_convergence}, up to subsequences,
\[\nu^{\Dx}\weakto \nu,\]
where $\nu$ is a EMVS of \eqref{eq:hyperbolic_conservation_law}. The rest of the proof concentrates on the stochastic error. We set 
\[G(\omega):=\int_{\DT}\psi(z)g(u^{\Dx}(\omega, z))\dd z\]
and
\[G_k(\omega):=\int_{\DT}\psi(z)g(u_k^{\Dx}(\omega, z))\dd z.\]
We have
\[\weakdiffpg{\nu^{\Dx}}{\FKMT{M}{u}{\Dx}}=\E(G)-\frac{1}{M}\sum_{k=1}^M G_k(\omega),\]
applying \Cref{lem:mc} yields the desired result.
\end{proof}
\begin{remark}
	From the approximate measure $\FKMT{\Dx}{u_0}{M}$, one can compute the statistics through evaluating the integral $\int_{\Ph} g(\xi)\dd \FKMT{\Dx}{u_0}{M}$. For the expectation, we have
	\[\E(\FKMT{\Dx}{u_0}{M})=\int_{\Ph}\xi \dd \FKMT{\Dx}{u_0}{M}=\frac{1}{M}\sum_{k=1}^{M}u_k^{\Dx}.\]
	A similar expression can be derived for the variance.
\end{remark}

\subsection{Atomic initial data}
	Even though the initial measure $\sigma$ may be \emph{atomic}, in other words $\sigma_x = \delta_{u_0(x)}$ for some $u\in L^1(\D, \Ph)$, the entropy measure valued solutions of \eqref{eq:hyperbolic_conservation_law} may be non-atomic. However, Algorithm \ref{alg:fkmt} will produce an atomic solution in this case. 
	
	The technique developed in \cite{fkmt} is to perturb the initial data by a small random variable. The following theorem makes this precise
	
	\begin{theorem}[Theorem 4.7 \cite{fkmt}]
		\label{thm:non_atomic}
	 Let $X:\Omega\to L^1(\D, \Ph)\cap L^\infty(\D, \Ph)$ be a random field such that $\|X\|\leq 1$ $\Prb$-almost surely, and let $\epsilon,\Dx>0$. Set $\sigma^\epsilon=\delta_{u_0}+\epsilon X$, and choose $u^\epsilon_0$ to be a random field with law $\sigma^\epsilon$. Set 
	 \[u^{\Dx,\epsilon}=\NumericalEvolution{\Dx}(u^\epsilon_0),\]
	and let $\nu^{\Dx,\epsilon}$ be the law of $u^{\Dx,\epsilon}$. 
	Then there exists a subsequence $(\Dx_n, \epsilon_n)\to 0$, such that 
	\[\nu^{\Dx_n,\epsilon_n}\weakto \nu\]
	where $\nu$ is an entropy measure valued solution of \eqref{eq:hyperbolic_conservation_law}.
	
\end{theorem}

\begin{remark}
	Based on \Cref{thm:non_atomic}, we will for the rest of the paper always assume the initial data is non-atomic.
	\end{remark}
	
\subsection{Work analysis for the FKMT algorithm}
\label{subsec:work_analysis_mc}
The work of the numerical method is given as the number of floating point operations it consumes. The classical explicit finite volume method has a work estimate of
\begin{equation}
\WorkFVM{\Dx}{\Dt} = \bigO(\Dx^{-d}\Dt^{-1}).
\end{equation}
Applying the CFL requirement $\Dt= c\Dx$, gives 
\[\WorkFVM{\Dx}{\Dt} = \WorkFVMCFL{\Dx} = \bigO(\Dx^{-d-1}).\]
Thus, the work to compute $\FKMT{\Dx}{u_0}{M}$ scales as
\[\WorkMC{\Dx}{M}=M\WorkFVMCFL{\Dx}=\bigO(M\Dx^{-d-1}).\]
If we assume the spatial narrow convergence error scales as
\[\Ypair{\psi}{\Ypair{\nu-\nu^{\Dx}}{g}}=\bigO(\Dx^s)\qquad \text{for all } \psi\in L^1(\D), g\in C_b(\Ph),\]
we choose the number of samples such that the Monte-Carlo error is asymptotically the same as the spatial error. That is, we choose
\[M^{-1/2}=\bigO(\Dx^s)\Rightarrow M=\bigO(\Dx^{-2s}).\]
This gives the work estimate
\begin{equation}
\label{eq:work_mc}
\WorkMC{\Dx}{M}=\bigO(\Dx^{-d-1-2s}).
\end{equation}

\section{Multilevel Monte Carlo}
\label{sec:mlmc}
The FKMT algorithm has shown great robustness for computing measure valued solutions of \eqref{eq:measure_valued_conservation_law}, but as it made clear by the work estimate \eqref{eq:work_mc} it suffers from the high computational cost of the Monte-Carlo algorithm. It is therefore appealing to study the behavior of alternative, faster stochastic methods. 

Inspired by the FKMT algorithm \cite{giles} and the MLMC method for conservation laws \cite{mlmc_hyperbolic}, we construct the multilevel Monte-Carlo algorithm for computing measure valued solutions for conservation laws.

We assume we have a nested collection of uniform Cartesian meshes $\{\Mesh_l\}_{l=0}^L$ with associated mesh widths $\{\Dx_l\}_{l=0}^L$, where 
\[\Dx_l=2^{-l}\Dx_0\qquad \text{for } l> 0,\] 
and $\Dx_0$ is some given parameter. For each level $l=0, \ldots, L$, we set 
\[u^l:=\NumericalEvolution{\Dx_l}(u_0).\]
By simply canceling terms, we have
\[u^L=\sum_{l=1}^L\left(u^l-u^{l-1}\right)+u^0,\]
which motivates the MLMC algorithm:

%

\begin{Algorithm}[MLMC]
	\label{alg:mlmc}
	Let $\sigma \in\Y(\D, \Ph)$ be the initial data, and choose a probability space $(\Omega, \Sigma, \Prb)$ together with $u_0:\Omega\times \D\to\Ph$ such that $\sigma_x=\Law{u_0(\cdot, x)}$. Let $L\in\N$ and $\{M_l\}_{l=0}^L\subset\N$.
	\begin{enumerate}
		\item Draw $M_0$ independent samples $\{u_0^{k,0}\}_k$ of $u_0$.
		\item Evolve the samples
		\[u^{\Dx_0}_k = \NumericalEvolution{\Dx_0}(u_0^{k,0})\]
		\item For $l=1,\ldots, L$:
		\begin{enumerate}
			\item Draw $M_0$ independent samples $\{u_0^{k,l}\}_k$ of $u_0$.
			\item Evolve the samples
			\[u^{\Dx_l,+}_k = \NumericalEvolution{\Dx_l}(u_0^{k,l})\]
			and
			\[u^{\Dx_{l-1},-}_k = \NumericalEvolution{\Dx_{l-1}}(u_0^{k,l})\]
		\end{enumerate}
		\item Estimate the measure
		\[\MLMC{\{M_l\}_{l=0}^L}{\Dx_0}{u_0}:=\frac{1}{M_0}\sum_{k=1}^{M_0} \delta_{u^{\Dx_0}_k}+\sum_{l=1}^L\frac{1}{M_l}\sum_{k=1}^{M_l}\left(\delta_{u^{\Dx_l,+}_k}-\delta_{u^{\Dx_{l-1},-}_k}\right) \]
	\end{enumerate}
\end{Algorithm}
Choosing the number of samples per level, $M_l$, depends on the exact error estimate we obtain for the MLMC algorithm. In the next section, we obtain an error rate for MLMC that can be used to determine the number of samples per level.
\subsection{Convergence analysis of MLMC}
\begin{theorem}[Weak convergence of MLMC]

\label{thm:weak_mlmc}
Let $\sigma\in \Y(\D,\Ph)$ and let $\MLMC{\{M_l\}_{l=0}^L}{\Dx_0}{u_0}$ be generated by \algref{alg:mlmc}, let $g\in C_b(\Ph)$ and $\psi\in L^1(\D)$, then 
\begin{multline}
\label{eq:mlmc_weak_variance_convergence}
\StochasticLTwoNorm{\Ypair{\psi}{\Ypair{\nu-\MLMC{\{M_l\}_{l=0}^L}{\Dx_0}{u_0}}{g}}}\leq\left|\Ypair{\psi}{\Ypair{\nu-\nu^{\Dx_L}}{g}}\right|+ \frac{\sqrt{\Var(\Ypair{\psi}{g(u^{0})})}}{M_0^{1/2}}\\+\sum_{l=1}^L\frac{\sqrt{\Var(\Ypair{\psi}{g(u^{l})-g(u^{l-1})})}}{M_l^{1/2}},
\end{multline}
where $\nu\in \Y(\DT, \Ph)$ is a entropy measure valued solution of \eqref{eq:measure_valued_conservation_law}, such that, up to a subsequence,
\[\nu^{\Dx_L}\weakto\nu.\]
Furthermore, if the samples in the Monte-Carlo sampling are chosen independently across levels, in other words, if $u_0^{k,l}$ and $u_0^{k,l'}$ are independent for $l\neq l`$, the stochastic error is sharp, i.e.
\begin{multline*}\label{eq:mlmc_sharp_stochastic}
\StochasticLTwoNorm{\weakdiffpg{\nu^{\Dx_L}}{\MLMC{\{M_l\}_{l=0}^L}{\Dx_0}{u_0}}}^2=\frac{\Var(\Ypair{\psi}{g(u^{0})})}{M_0}\\+\sum_{l=1}^L\frac{\Var(\Ypair{\psi}{g(u^{l})-g(u^{l-1})})}{M_l}
\end{multline*}
\end{theorem}
\begin{proof}
Let 
\[\nu^{\Dx_L}:=\Law(u^{\Dx_L}).\]
By \cref{thm:emvs_spatial_convergence}, we know that up to a subsequence, 
\begin{equation}
\label{eq:finest_level_conv}
\nu^{\Dx_L}\weakto \nu
\end{equation}
where $\nu$ is a entropy measure valued solution to \eqref{eq:measure_valued_conservation_law}. A simple application of the triangle inequality, splitting up the error in a spatial term and a stochastic term, yields
\begin{multline*}\StochasticLTwoNorm{\Ypair{\psi}{\Ypair{\nu-\MLMC{\{M_l\}_{l=0}^L}{\Dx_0}{u_0}}{g}}}\leq \StochasticLTwoNorm{\Ypair{\psi}{\Ypair{\nu-\nu^{\Dx_L}}{g}}}\\+\StochasticLTwoNorm{\Ypair{\psi}{\Ypair{\nu^{\Dx_L}-\MLMC{\{M_l\}_{l=0}^L}{\Dx_0}{u_0}}{g}}}.
	\end{multline*}
Clearly, 
\[ \StochasticLTwoNorm{\Ypair{\psi}{\Ypair{\nu-\nu^{\Dx_L}}{g}}}=\left|\Ypair{\psi}{\Ypair{\nu-\nu^{\Dx_L}}{g}}\right|,\]
and we are left with estimating the stochastic error. To this end, we insert for $\MLMC{\{M_l\}_{l=0}^L}{\Dx_0}{u_0}$ and write out the telescoping sum 
$$
\nu^{\Dx_L} = \sum_{l=1}^L\left(\nu^{\Dx_l}-\nu^{\Dx_{l-1}}\right)+\nu^{\Dx_0},
$$
to obtain
\begin{multline*}
\StochasticLTwoNorm{\Ypair{\psi}{\Ypair{\nu^{\Dx_L}-\MLMC{\{M_l\}_{l=0}^L}{\Dx_0}{u_0}}{g}}}\leq \StochasticLTwoNorm{\weakdiffpg{\nu^{\Dx_0}}{\frac{1}{M_0}\sum_{k=1}^{M_0} \delta_{u^{\Dx_0}_k}}}\\+\sum_{l=1}^L\StochasticLTwoNorm{\weakdiffpg{\nu^{\Dx_l}-\nu^{\Dx_{l-1}}}{\sum_{k=1}^{M_l}\left(\delta_{u^{\Dx_l,+}_k}-\delta_{u^{\Dx_{l-1},-}_k}\right)}}.
\end{multline*}
We appeal to \eqref{eq:fkmt_error_sharp} to obtain
\[\StochasticLTwoNorm{\weakdiffpg{\nu^{\Dx_0}}{\FKMT{M_0}{u_0}{\Dx_0}}}= \WeakErrorMC{u^{0}}{M_0^{1/2}}.\]
For $l=1,\ldots, L$ we set
\[G_k^l(\omega):=\int_{\DT}\psi(z) \left( g(u^{\Dx_l,+}_k)-g(u^{\Dx_{l-1},-}_k)\right)\dd z,\]
and
\[G(\omega):=\int_{\DT}\psi(z) \left( g(u^{\Dx_l,+})-g(u^{\Dx_{l-1},-})\right)\dd z.\] 
Applying \Cref{lem:mc}, we get
\[\StochasticLTwoNorm{\weakdiffpg{\nu^{\Dx_l}-\nu^{\Dx_{l-1}}}{\left(\sum_{k=1}^{M_l}\left(\delta_{u^{\Dx_l,+}_k}-\delta_{u^{\Dx_{l-1},-}_k}\right)\right)}}=  \WeakErrorDiffMC{u^{l}}{u^{l-1}}{M_l^{1/2}}.\]
Combining these estimates we obtain \eqref{eq:mlmc_weak_variance_convergence}.

The last assertion	 is obtained by noting that if the samples are chosen independently across levels, we have
\begin{multline*}
\StochasticLTwoNorm{\Ypair{\psi}{\Ypair{\nu^{\Dx_L}-\nu^{\Dx,L}}{g}}}^2= \StochasticLTwoNorm{\weakdiffpg{\nu^{\Dx_0}}{\FKMT{M_0}{u}{\Dx_0}}}^2\\+\sum_{l=1}^L\StochasticLTwoNorm{\weakdiffpg{\nu^{\Dx_l}-\nu^{\Dx_{l-1}}}{\sum_{k=1}^{M_l}\left(\delta_{u^{\Dx_l,+}_k}-\delta_{u^{\Dx_{l-1},-}_k}\right)}}^2,
\end{multline*}
 applying the same estimates yields the sharp bound.
\end{proof}

\begin{remark}
	As with the Monte-Carlo approach, one can compute the statistics through evaluating the integral $\int_{\Ph} g(\xi)\dd \MLMC{\{M_l\}_{l=0}^L}{\Dx_0}{u_0}$. For the expectation, we have
	\begin{multline*}\E(\MLMC{\{M_l\}_{l=0}^L}{\Dx_0}{u_0})=\int_{\Ph}\xi \dd \MLMC{\{M_l\}_{l=0}^L}{\Dx_0}{u_0}=\frac{1}{M_0}\sum_{k=1}^{M}u_k^{\Dx_0}\\+\sum_{l=1}^L\frac{1}{M_l}\sum_{k=1}^{M_l}\left(u^{\Dx_l}_k-u^{\Dx_{l-1}}_k\right).
	\end{multline*}
	A similar expression can be derived for the variance.
\end{remark}

\begin{remark}
	The theorem above involves a priori unknown functions $g$ and $\psi$. In practical computational examples, $g$ and $\psi$ are known as they are given through the statistics, and we can calculate a concrete error estimate for the given $g$ and $\psi$. 
\end{remark}

\subsection{Work analysis of MLMC}
We extend the analysis of \Cref{subsec:work_analysis_mc} to the MLMC algorithm.

To compute $\MLMC{\{M_l\}}{\Dx}{u_0}$, we compute $M_l$ finite volume simulations with resolution $\Dx_l$ for each $l\geq 0$, and $M_l$ finite volume simulations with resolution $\Dx_{l-1}$ for $l>0$. Since the latter can be neglected, we obtain
\begin{equation}
\label{eq:work_mlmc}
\begin{aligned}
\WorkMLMC{\{M_l\}}&= \sum_{l=0}^LM_l\WorkFVMCFL{\Dx_l}\\
&=\sum_{l=0}^L\bigO(M_l(\Dx^{-d-1}_l))\\
&=\sum_{l=0}^L\bigO(M_l2^{l(d+1)}\Dx_0^{-d-1}).
\end{aligned}
\end{equation} 

\subsection{Choosing optimal number of samples}
\label{subsec:optimal_samples}
The number of samples per level, $M_l$, has so far been unspecified. It is common to optimize the number of samples for a given convergence rate. We handle the general case, and optimize with respect to the number of samples, where the variance across the levels is abstractly given as
\begin{equation}
\label{eq:variance_v_l}
\Var(\Ypair{\psi}{g(u^{\Dx_l})-g(u^{\Dx_{l-1}})})=: V_l.
\end{equation}

We furthermore define the \emph{asymptotic speedup} between two asymptotic work estimates $W_1$ and $W_2$ as
\[\SpeedUp(W_1,W_2):=\frac{W_2}{W_1}.\]

\begin{theorem}
	\label{thm:number_of_samples}
	Assume the variance between levels is given by \eqref{eq:variance_v_l}. For a given $L>0$, the optimal number of samples per level, $M_l$, to ensure that 
	\[\frac{1}{M_0^{1/2}}+\sum_{l=1}^L\frac{\sqrt{V_l}}{M_l^{1/2}}\leq \tau,\]
	and minimizing the work \eqref{eq:work_mlmc},
	is given as
	\[M_0^{1/2}=\frac{1}{\tau}\left(1+\sum_{l=1}^L2^{(l(d+1))/3}V_l^{1/3}\right).\]
	and
	\[M_l^{1/2}=\frac{V_l^{1/6}\left(1+\sum_{l=1}^L2^{(l(d+1))/3}V_l^{1/3}\right)}{\tau 2^{(l(d+1))/3}}.\]
	Furthermore, if 
	\[V_l=\bigO(\Dx_l^{q}),\]
	then 
	\[\WorkMLMC{\{M_l\}_l}< \WorkMC{\tau^{-2}}{\Dx_L}\Leftrightarrow q>0.\]
	Setting $\tau=\bigO(\Dx_L^s)$ we get
	\[\WorkMLMC{\{M_l\}_l} = \bigO(\Dx_L^{-2s-d-1}/2^{qL}),\]
	and correspondingly
	\[\SpeedUp(\WorkMLMC{\{M_l\}_l}, \WorkMC{\Dx^{-2s}}{\Dx_L})=\bigO(2^{-qL}).\]
	
\end{theorem}

\begin{proof}
	In this proof, we let $A\simeq B$ denote $\bigO(A)=\bigO(B)$ and $A\prec B$ denote $\bigO(A)<\bigO(B)$.
	For a given $L>0$, we solve the following optimization problem
	\begin{equation}
	\begin{aligned}
	&\min \WorkMLMC{\{M_l\}_{l=0}^L}=\min \sum_{l=0}^L M_l2^{l(d+1)}\Dx_0^{-d-1}\\
	\text{s.t.} \qquad & \frac{1}{M_0^{1/2}}+\sum_{l=1}^L\frac{\sqrt{V_l}}{M_l^{1/2}}\leq \tau.
	\end{aligned}
	\end{equation}
	We use a Lagrange multiplier technique, and introduce the function
	\[H(M_0, \ldots, M_L, \lambda) = \sum_{l=0}^L M_l2^{l(d+1)}\Dx_0^{-d-1}+ \lambda \left (\frac{1}{M_0^{1/2}}+\sum_{l=1}^L\frac{\sqrt{V_l}}{M_l^{1/2}} - \tau \right). \]
	The extremal point must obey
	\[\pdpd{H}{\lambda}=\pdpd{H}{M_l} =0.\]
	We readily compute 
	\[
	\pdpd{H}{M_l} = \begin{cases}\Dx_0^{-d-1}-\lambda \frac{1}{2M_0^{3/2}} &\text {if } l=0 \\
	-\lambda\frac{\sqrt{V_l}}{2M_l^{3/2}}+
	2^{l(d+1)}\Dx_0^{-d-1}&\text{otherwise.}
	\end{cases}
	\]
	From $\pdpd{H}{M_0}=0$, we get 
	\[\lambda = \Dx_0^{-d-1}2M_0^{3/2}.\]
	For $l>0$ we solve for $\pdpd{H}{M_l}$ for $M_l$ to get
	\[M_l^{1/2}=\frac{V_l^{1/6}M_0^{1/2}}{2^{(l(d+1))/3}}.
	\]
	Solving $\pdpd{H}{\lambda}=0$ for $M_0$ gives us
	\[M_0^{1/2}=\frac{1}{\tau}\left(1+\sum_{l=1}^L2^{(l(d+1))/3}V_l^{1/3}\right).\]
	Inserting this for $M_l$ gives 
	\[M_l^{1/2}=\frac{V_l^{1/6}\left(1+\sum_{l=1}^L2^{(l(d+1))/3}V_l^{1/3}\right)}{\tau 2^{(l(d+1))/3}}.\]
	The work estimate is then
	\begin{align*}
	\WorkMLMC{\{M_l\}_{l=0}^L} \simeq \tau^{-2}\Dx_0^{-d-1}\left(1+\sum_{l=1}^L2^{(l(d+1))/3}V_l^{1/3}\right)^2\\ \left[1  + \sum_{l=1}^L V_l^{1/3}2^{l(d+1)/3}\right].
	\end{align*}
	For the last assertion, insert
	\[V_l=\Dx_l^q=\Dx_0^q2^{-lq},\]
	to see that
	\[\left(1+\sum_{l=1}^L2^{(l(d+1))/3}V_l^{1/3}\right)\simeq\sum_{l=1}^L V_l^{1/3}2^{l(d+1)/3} =\sum_{l=1}^L 2^{-lq/3}2^{l(d+1)/3} \simeq 2^{\left[(d+1-s)/3\right]L}.\]
	Therefore,
	\begin{align*}
	\WorkMLMC{\{M_l\}_{l=0}^L}&\simeq \tau^{-2}\Dx_0^{-d-1}2^{\left[(d+1-q)\right]L}\\&\prec \tau^{-2}\Dx_L^{-d-1}&\Leftrightarrow q>0\\
	&
	=\WorkMC{\tau^{-2}}{\Dx_L}.
	\end{align*}
	The last work estimate is then found by insertion.
\end{proof}

\subsection{Scalar conservation laws}
In the case of a scalar conservation law, we can appeal to the readily available sample convergence of the numerical scheme to produce an error estimate close to the MLMC error estimate found in \cite{mlmc_hyperbolic}. Here we assume that our scheme is able to reproduce the exact solution up to an order $s$. In other words, we assume
\begin{equation}
\label{eq:scalar_det_convergence}
\|u(\cdot, t)-\NumericalEvolution{\Dx}(u_0)(\cdot, t)\|_{L^1(\D)}\leq C\Dx^s,
\end{equation}
where $u$ is the unique, exact solution of \eqref{eq:hyperbolic_conservation_law}. For scalar conservation laws, such schemes are readily available, see for instance \cite{tecno} and \cite{hiltebrand_stdg}.

\begin{corollary}[MLMC for Scalar Conservation Laws]
\label{thm:mlmc_scalar}
Assume $N=1$, and let $\sigma\in \Y(\D,\Ph)$, and $\MLMC{\{M_l\}_{l=0}^L}{\Dx_0}{u_0}$ be generated by \algref{alg:mlmc}, let $g\in C_b(\Ph)\cap \text{Lip}(\Ph)$ and $\psi\in L^1(\D)\cap L^\infty(\D)$, and assume the numerical evolution operator satisfies \eqref{eq:scalar_det_convergence}, then 
\begin{multline}
\label{eq:mlmc_scalar_weak_error}
\StochasticLTwoNorm{\Ypair{\psi}{\Ypair{\nu-\MLMC{\{M_l\}_{l=0}^L}{\Dx_0}{u_0}}{g}}}\leq\left|\Ypair{\psi}{\Ypair{\nu-\nu^{\Dx_L}}{g}}\right|\\+C_1 \frac{\sqrt{\Var(\Ypair{\psi}{g(u^{0})})}}{M_0^{1/2}}+C_2\sum_{l=1}^L\frac{\|\psi\|_{L^\infty(\D)}\|g\|_{\text{Lip}}\Dx_l^s}{M_l^{1/2}}.
\end{multline}
\end{corollary}
\begin{proof}
By \Cref{thm:weak_mlmc}, the only thing we need to show is
\[\Var(\Ypair{\psi}{g(u^{l})-g(u^{l-1})})\leq C\|\psi\|^2_{L^\infty(\D)}\|g\|^2_{\text{Lip}}\Dx_l^{2s}.\]
Since each sample will converge, we readily estimate
\[\int_{\DTT}\psi(x,t)(g(u^{l})-g(u^{l-1}))\dd x \dd t\leq \|\psi\|_{L^\infty(\DTT)}\|g\|_{\Lip}\|u^{l}-u^{l-1}\|_{L^1(\DTT)}.\]
Owing to \eqref{eq:scalar_det_convergence} and the triangle inequality, we see that
\[\|\psi\|_{L^\infty(\DTT)}\|g\|_{\Lip}\|u^{l}-u^{l-1}\|_{L^1(\DTT)}\leq C\|\psi\|_{L^\infty(\DTT)}\|g\|_{\Lip}\Dx_l^s.\]
Now, we easily obtain
\begin{align*}
\Var(\Ypair{\psi}{g(u^{l})-g(u^{l-1})})&=\E(\Ypair{\psi}{g(u^{l})-g(u^{l-1})}^2)-\E(\Ypair{\psi}{g(u^{l})-g(u^{l-1})})^2\\
&\leq\int_\Omega C\|\psi\|^2_{L^\infty(\DTT)}\|g\|^2_{\Lip}\Dx_l^{2s}\dd \Prb(\omega)\\
&=C\|\psi\|^2_{L^\infty(\DTT)}\|g\|^2_{\Lip}\Dx_l^{2s},
\end{align*}
taking square roots gives the claim.
\end{proof}
\begin{remark}
	In the above theorem, we had to put restrictions on $g$ and $\psi$. This is only needed to give known bounds on $V_l$. Indeed, owing to the dominted convergence theorem, $V_l\to 0$ for any $g\in C_b(\Ph)$ and $\psi \in L^1(\DT)$, but not with a necessarily with a computeable decay rate.
\end{remark}
\subsection{Systems of conservation laws}

For systems of conservation laws, we can not appeal to any convergence result for a numerical scheme as we did for scalar conservation laws. However, we can measure the decay of the variance between the levels, $V_l$, numerically. 

\section{Numerical experiments}
\label{sec:numerical_experiments}
We perform numerical experiments to assess the applicability of MLMC for entropy measure valued solutions.
\subsection{Scalar conservation laws}
Owing to \eqref{eq:mlmc_scalar_weak_error} and the optimal work estimates derived in \Cref{thm:number_of_samples}, in the scalar case we already expect that the MLMC method will provide a speedup compared to ordinary Monte-Carlo.

In this subsection, we consider the Burgers equation in one space dimension, given here as
\begin{equation}                                                                                                                                                      
\label{eq:burgers}                                                                                                                                                    
u_t+\left[\frac{u^2}{2}\right]_x=0.                                                                                                                                             
\end{equation}
Here $u:\R\times \TD \to \R$ is the unknown. We consider the initial data
\begin{equation}
\label{eq:init_dirac_burgers}
u_0(x,\omega)=\begin{cases}1& x <1/2+\epsilon X(\omega)\\
0 & \text{otherwise}\end{cases}\qquad x \in[0,1],
\end{equation}
where $X$ is uniformly distributed on $[-0.5, 0.5]$. We pick the number of samples in accordance with \Cref{thm:number_of_samples} and \Cref{thm:mlmc_scalar}. Here $s\approx 1$.

We measure the convergence against the Dirac solution when $\epsilon\to 0$, given as $\delta_{u(x,t)}$, where 
\[u(x,t) = \begin{cases}1& x <1/2 + t\\
0 & \text{otherwise}\end{cases}\qquad x \in[0,1],\]
and we simulate to $t=0.1$. The results are shown in \Cref{fig:burgers_dirac}. As is expected, the convergence rate of the MLMC algorithm is linear with respect to the runtime, while the Monte-Carlo algorithm scales as $\bigO(\epsilon^{-3})$.

\begin{figure}[h!]
	
	\includegraphics[width=\textwidth]{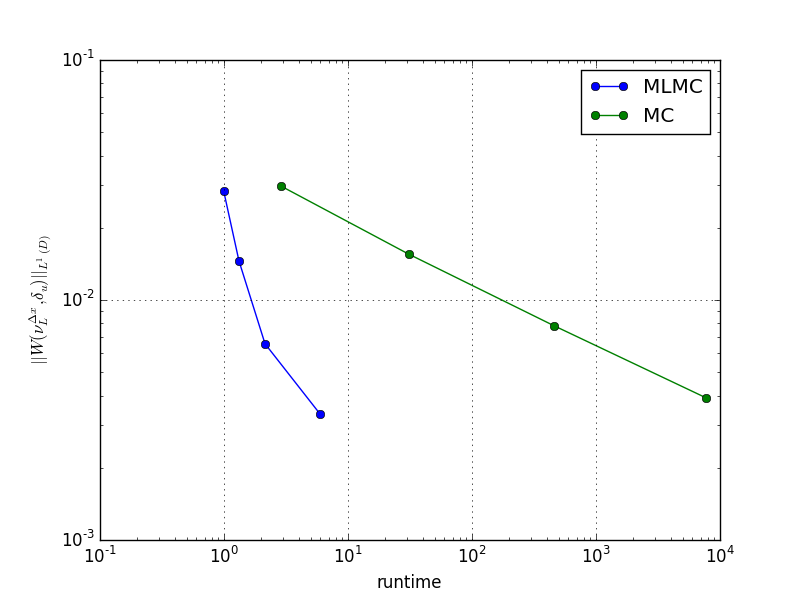}
	\caption{\label{fig:burgers_dirac}Wasserstein convergence, comparing MC with MLMC, with initial data \eqref{eq:init_dirac_burgers}}.
\end{figure} 

\begin{figure}[h!]
	\begin{subfigure}[b]{0.45\textwidth}
		\includegraphics[width=\textwidth]{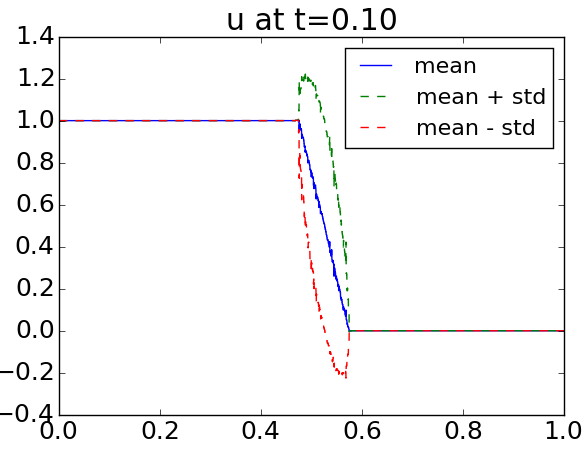}
		\caption{Variance and mean computed by MLMC.}
	\end{subfigure}
	\begin{subfigure}[b]{0.45\textwidth}
		\includegraphics[width=\textwidth]{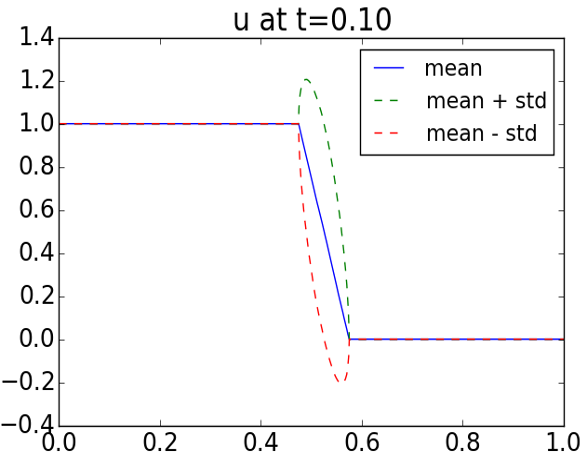}
		\caption{Variance and mean from reference solution.}
	\end{subfigure}
	\caption{MLMC computed results for the Riemann problem \eqref{eq:init_dirac_burgers}.}
\end{figure}

\subsection{System of conservation laws}

We consider the Euler equations, given here as
\begin{equation}                                                                                                                                                      
\pdpd{}{t}\begin{pmatrix}                                                                                                                                               
\rho\\                                                                                                                                                                
\rho w_x\\                                                                                                                                                            
\rho w_y\\                                                                                                                                                            
\rho E                                                                                                                                                                
\end{pmatrix}                                                                                                                                                         
+                                                                                                                                                                     
\pdpd{}{x_1}                                                                                                                                                            
\begin{pmatrix}                                                                                                                                                       
\rho w_x\\                                                                                                                                                            
\rho w_x^2+p\\                                                                                                                                                        
\rho w_yw_y\\                                                                                                                                                         
\rho (E+p)w_x                                                                                                                                                         
\end{pmatrix}                                                                                                                                                         
+                                                                                                                                                                     
\pdpd{}{x_2}                                                                                                                                                            
\begin{pmatrix}                                                                                                                                                       
\rho w_y\\                                                                                                                                                            
\rho w_xw_y\\                                                                                                                                                         
\rho w_y^2+p\\                                                                                                                                                        
\rho (E+p)w_y                                                                                                                                                         
\end{pmatrix}                                                                                                                                                         
=0.                                                                                                                                                                   
\end{equation}
Here the pressure $p$, the density $\rho$, the total energy $E$ and the velocity field $(w^x, w^y)$ are related through
\[E=\frac{p}{\gamma-1}+\frac{\rho\left(w_x^2+w_y^2\right)}{2},\]
where $\gamma$ is the adiabatic constant, which we set to $1.4$.
\subsubsection{Shockvortex interaction}
We consider the initial data
\begin{equation}
\label{eq:init_sv}
u_0(\omega, x)=\begin{cases}
u_L&x_1<I\\
u_R&\text{otherwise.}\end{cases}\qquad x\in [0,1]^2
\end{equation}
with $\rho_L=2$, $\rho_R=1/1.1$, 
\[w^x_L=\sqrt{\gamma}+\delta b \exp(\alpha(1 - b^2)) \sin(\theta),\]
\[w^y_L=\sqrt{\gamma}-\delta b \exp(\alpha(1 - b^2)) \cos(\theta),\]
\[p_L= 1 - (\gamma - 1) \delta^2 \frac{ \exp(2 \alpha(1 - b^2))}{ (4 \alpha \gamma)}\rho_L\]
$w^x_L=1.1\sqrt{\gamma}$,  $w^y_R=0$ and $p_R=1-0.1\gamma$. Here 
\[b= \frac{\sqrt{(x-0.25)^2 + (y-0.5)^2}}{0.05}\]
and $\theta$ is the angle between the $x$-axis and the line spanned by $(x-0.25, y-0.5)$. We set $\delta=0.3$. In addition, we perturb the interfaces $I$  by setting
\[I=0.5+\epsilon Y(x,\omega)\]
where $\epsilon>0$ will be a parameter to the simulation, and
\[Y(x, \omega)=\sum_{n=1}^m a^n(\omega)\cos(b^n(\omega)+2n\pi x_2).\]
We simulate to $T=0.35$. In the simulation, we set $m=10$ and $\epsilon=0.1$.

In lieu of \eqref{eq:mlmc_sharp_stochastic} and \Cref{thm:number_of_samples}, the MLMC algorithm will only give a computational speed up compared to Monte-Carlo if the variance between the samples decays with $\Dx^q$ for some $q>0$, in other words if
\[\Var(\Ypair{\psi}{g(u^l)-g(u^{l-1})})\leq C\Dx^q.\]

To measure the decay rate of the variance between the samples, we do a regular Monte-Carlo simulation to measure $V_l$. Concretely, we approximate
\begin{equation}
\label{eq:mc_variance_approx}
V_l\approx \frac{1}{M}\sum_{k=1}^M\left(\Ypair{\psi}{g(u^l_k)-g(u^{l-1}_k)}\right)^2-\left(\frac{1}{M}\sum_{k=1}^M\left(\Ypair{\psi}{g(u^l_k)-g(u^{l-1}_k)}\right)\right)^2.
\end{equation}

We display the result in \Cref{fig:sv_variance_levels}. In this case the variance actually decays with the levels, and the decay rate is close to $1$. Therefore, it is expected that the MLMC method works. We pick the number of samples per level in accordance  with the decay rate and \Cref{thm:number_of_samples}.

To verify the assertion in the previous paragraph, we perform numerical experiments with MLMC and regular single level Monte-Carlo. We compute a reference solution at resolution $1024\times 1024$ using $1000$ samples. In \Cref{fig:kh_mean_mlmc_compare}, we compare the errors of single level Monte-Carlo to that of multilevel Monte-Carlo using a varying amount of samples at the finest level. Our claims are confirmed in \Cref{fig:sv_mlmc_error}. As we can see, the MLMC method starts of with a low error even with a low number of samples on the highest level. With a higher number of samples, the Monte-Carlo algorithm eventually beats the MLMC algorithm, as is expected.

In \Cref{fig:sv_numerical} we display the results of the computation. As is clear, the MLMC algorithm works well for this initial data, since we actually do observe decay in the variance.
\begin{figure}[h!]
	
	\includegraphics[width=0.6\textwidth]{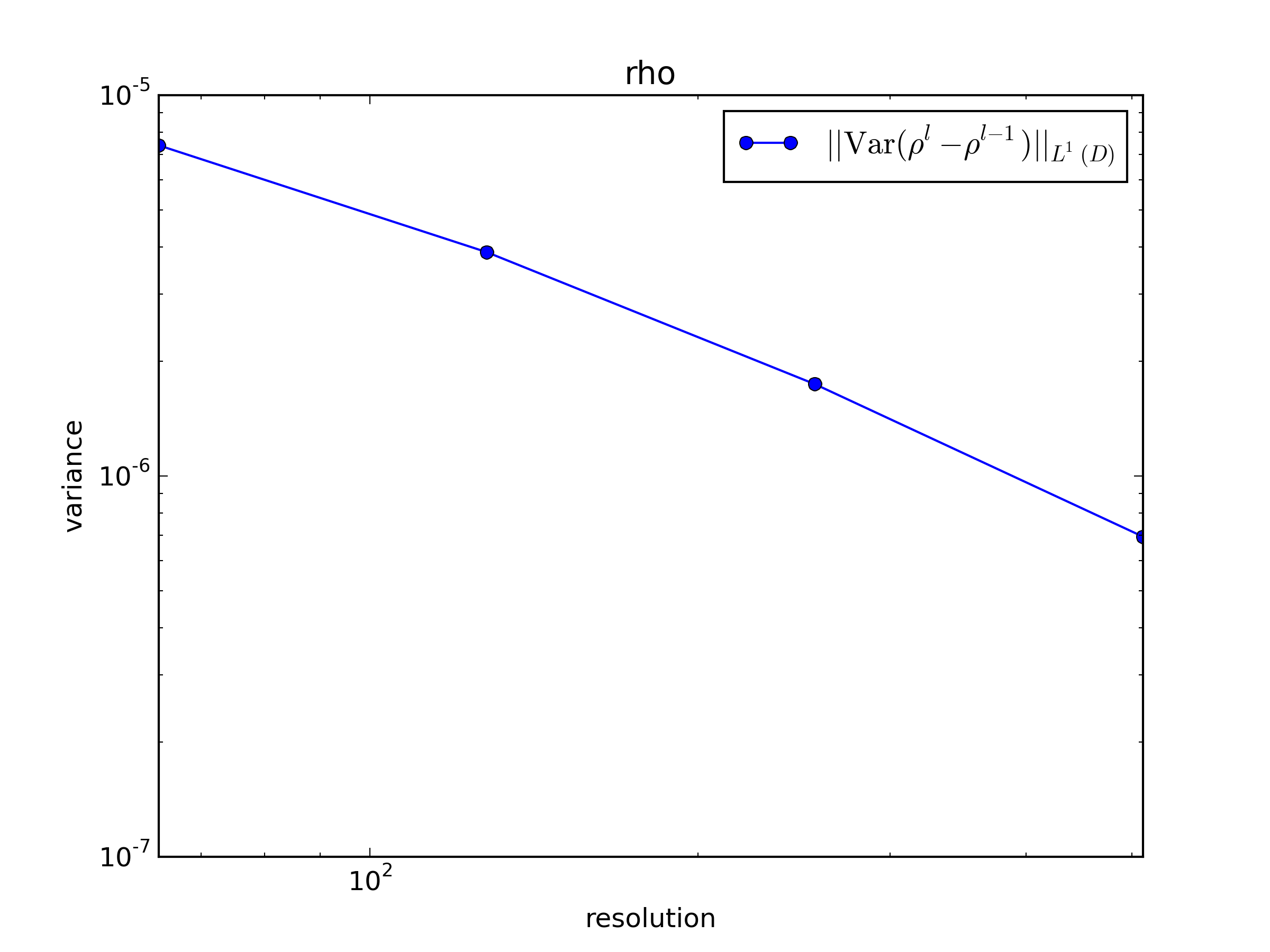}
	\caption{Measured $V_l$, the variance across levels, for the shock vortex simulation \eqref{eq:init_sv}  at time $T=0.35$.}
	
	\label{fig:sv_variance_levels}
\end{figure}

\begin{figure}[h!]
	\includegraphics[width=0.6\textwidth]{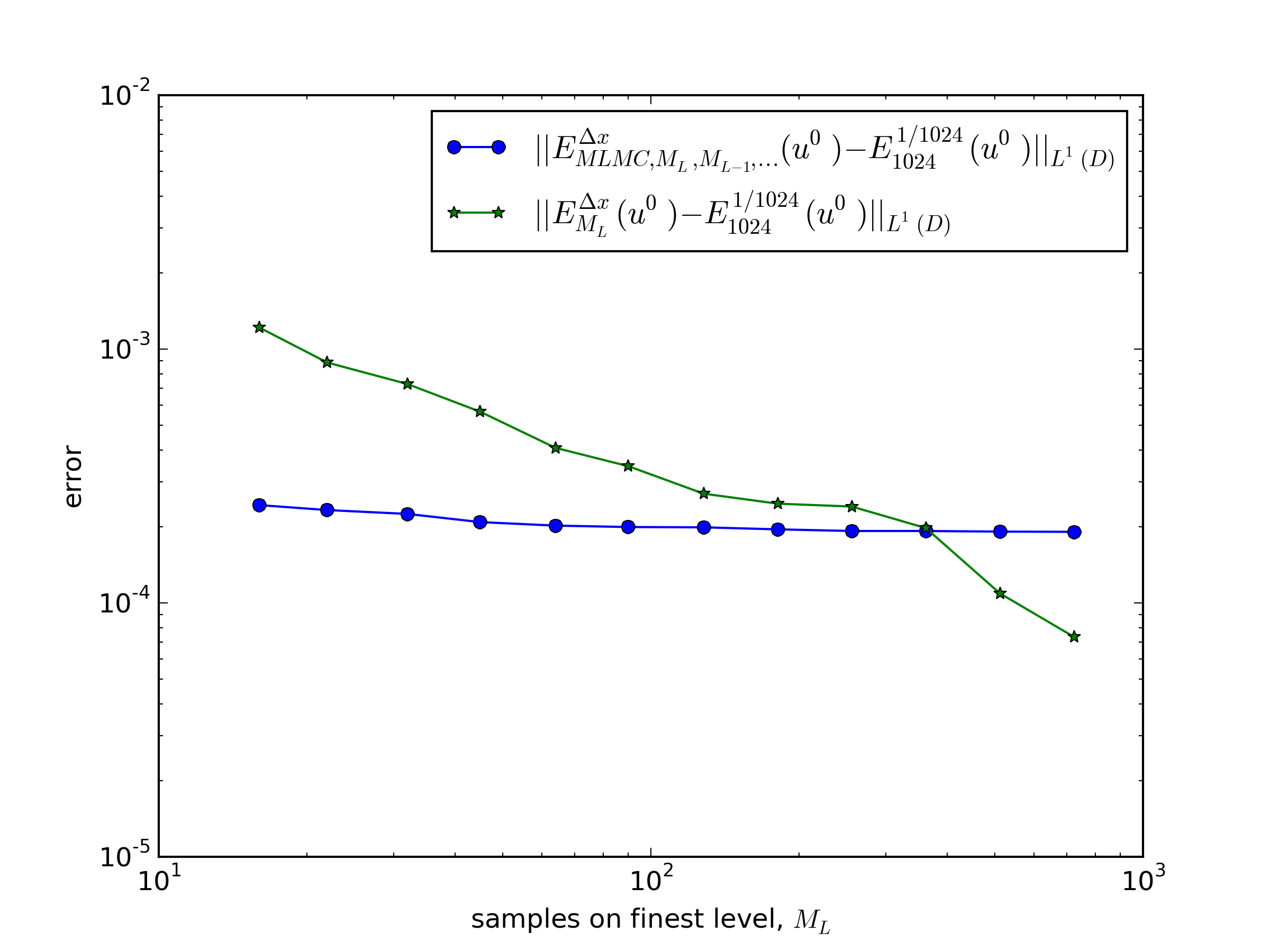}	
	\caption{Comparison of error of MLMC and MC for the shock vortex initial data \eqref{eq:init_sv} at time $T=0.35$.}
	
	\label{fig:sv_mlmc_error}
\end{figure}
\begin{figure}[h!]
	\begin{subfigure}[b]{0.45\textwidth}
		\includegraphics[width=\textwidth]{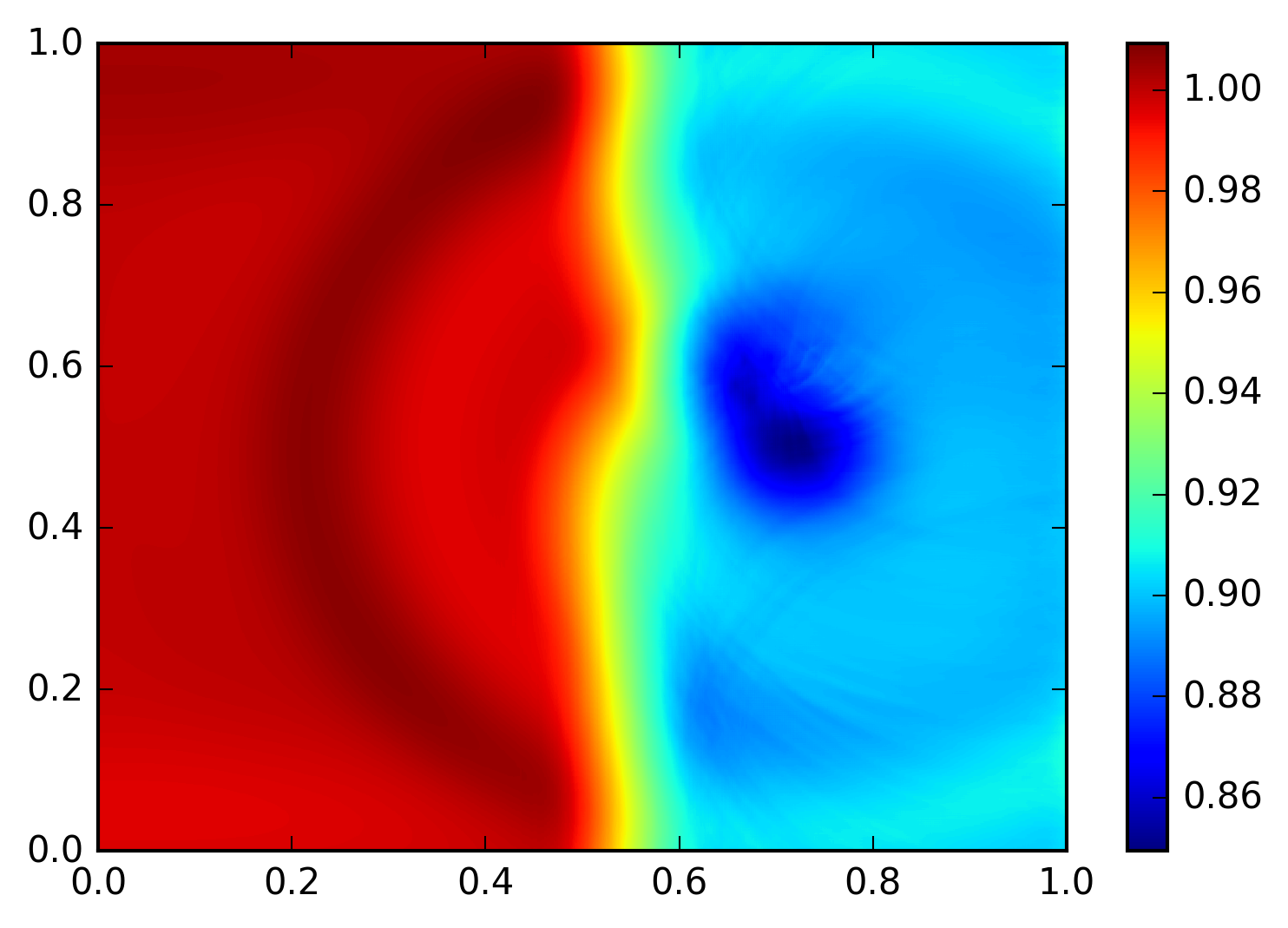}
		\caption{Mean from MLMC.}
	\end{subfigure}
	\begin{subfigure}[b]{0.45\textwidth}
		\includegraphics[width=\textwidth]{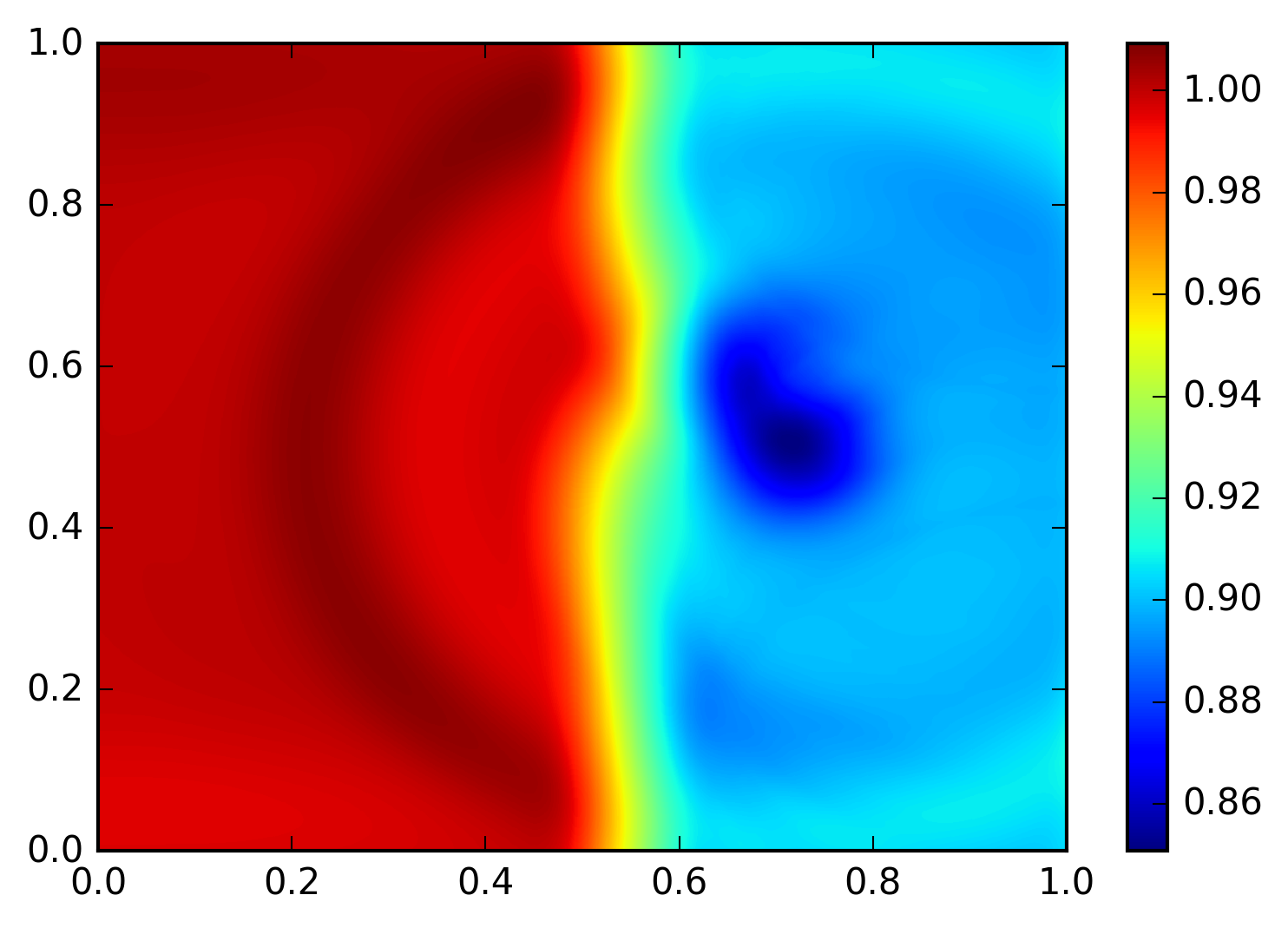}
		\caption{Mean from reference solution.}
	\end{subfigure}
	
	\begin{subfigure}[b]{0.45\textwidth}
		\includegraphics[width=\textwidth]{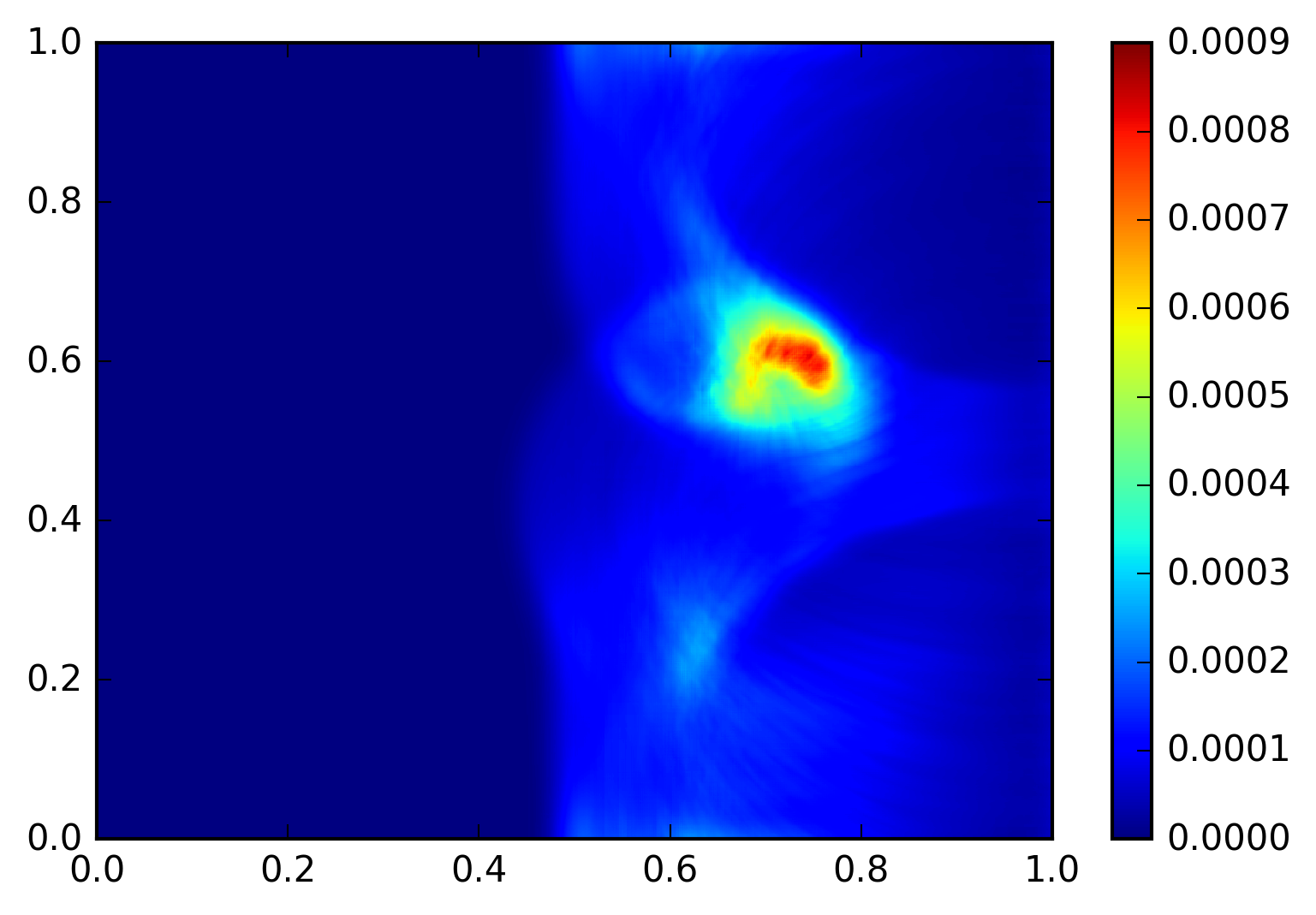}
		\caption{Variance from MLMC.}
	\end{subfigure}
	\begin{subfigure}[b]{0.45\textwidth}
		\includegraphics[width=\textwidth]{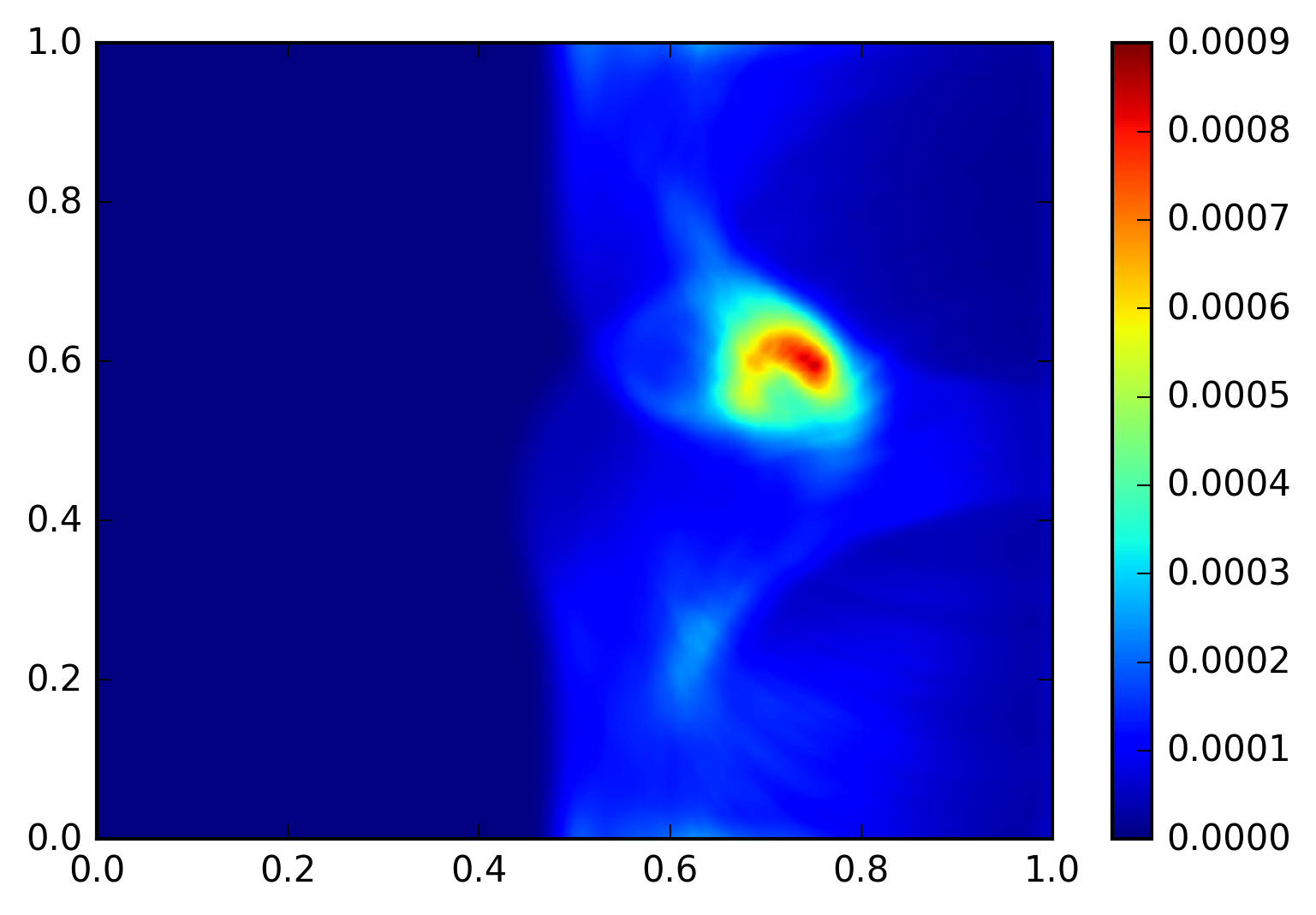}
		\caption{Variance from reference solution.}
	\end{subfigure}
	\caption{Comparison of the singlelevel Monte-Carlo algorithm against the MLMC algorithm for the shockvortex interaction initial data \eqref{eq:init_sv}  at time $T=0.035$. We used $M_L=16$ samples on the finest level for the MLMC computation..}
	\label{fig:sv_numerical}
\end{figure}

\subsubsection{Kelvin-Helmholtz initial data}

We use the initial data 
\begin{equation}
\label{eq:init_kh}
u_0(\omega, x)=\begin{cases}
u_L&I_1<x_2<I_2\\
u_R&\text{otherwise.}\end{cases}\qquad x\in [0,1]^2
\end{equation}
with $\rho_L=2$, $\rho_R=1$, $w^x_L=-0.5$, $w^x_R=0.5$, $w^y_L=w^y_R=0$ and $p_L=p_R=2.5$. In addition, we perturb the interfaces $I_1$ and $I_2$ by setting
\[I_j=J_j+\epsilon Y_j(x,\omega)\qquad j=1,2\]
where $\epsilon>0$ will be a parameter to the simulation, and
\[Y_j(x, \omega)=\sum_{n=1}^m a_j^n(\omega)\cos(b_j^n(\omega)+2n\pi x_1)\qquad j=1,2\]
for uniformly distributed random variables $a_j^n:\Omega\to [0,1]$ and $b_j^n:\Omega\to [0,2\pi]$. In \cite{fkmt}, numerical experiments indicated that no relevant numerical scheme was able to obtain sample convergence for this initial data. However, the FKMT algorithm did produce a numerical approximation that converged. We simulate to $T=2$. In the simulation, we set $m=10$ and $\epsilon=0.1$. We simulate using a 3-wave HLL solver~\cite{Toro1994} with third order WENO reconstruction~\cite{WENO}.

In \Cref{fig:kh_variance_levels}, we plot the numerically computed variance. What is immediately clear from the plot, is that the variance does not decrease in any significant way. Hence, we can not expect that MLMC will improve upon Monte-Carlo. There is also no observed variance decay for the functionals $\Ypair{\psi}{g(u^l)}$ for $\psi=1$ and $g$ being set as an $k$-order Legendre polynomial, for $k>1$.

To verify the assertion in the previous paragraph, we perform numerical experiments with MLMC and regular single level Monte-Carlo. We compute a reference solution at resolution $2048\times 2048$ using $2000$ samples. In \Cref{fig:kh_mean_mlmc_compare}, we compare the errors of single level Monte-Carlo to that of multilevel Monte-Carlo using a varying amount of samples at the finest level. The figure clearly shows that the MLMC algorithm, even with more work performed than the Monte-Carlo algorithm, is no better than the Monte-Carlo algorithm. Plots of the numerical results are shown in \Cref{fig:kh_numerical}. As is clear from the figures and theory, MLMC can not give a speed up compared to MC for the unstable Kelvin-Helmholtz initial data.
\subsubsection{MLMC with relaxation}
In the case of the Kelvin-Helmholtz equation, we do observe sample convergence for small times. That is, for $0\leq t\leq T_0=0.05$, we observe 
\begin{equation}
\label{eq:kh_short_sample_conv}
\|u^{\Dx_l}(\cdot, t)-u^{\Dx_{l-1}}(\cdot,t)\|_{L^1(\D, \Ph)}\leq C\Dx^s,
\end{equation}
for some $s\approx 1$. We can exploit this to try to correct the MLMC method by introducing a so-called relaxation time, described here. We fix $T_0\approx 0.05$, and then we reset the coarse samples with the fine samples for every $t=nT_0$. In other words, we run the simulation between $t=(n-1)T_0$ and $t=nT_0$, then we reset the coarse samples by
\[u^{\Dx_{l-1}}_k(\omega, x,nT_0)=u^{\Dx_{l}}_k(\omega, x,nT_0).\]

Since we observe short time sample convergence, this  guarantees that 
\[\Var\left(\Ypair{\psi}{g(u^l)-g(u^{l-1})}\right)\leq C\Dx^{2s},\]
as illustrated in \Cref{fig:mlmc_variance_stabilization}.
However, by resetting the coarse samples, we introduce an error term of the form
\[\sum_{l=0}^{L-1}|\langle \psi, \langle \nu^{\Dx_{l}} - \nu^{\Dx_{l-1}}, g\rangle\rangle|=\mathcal{O}(\Dx_0^s).\]
The error term is independent of the number of samples on each level, and scales as the coarsest resolution. This can clearly be seen in \Cref{fig:mlmc_vs_mc_samples_stabilization}. Also with the relaxation time, the MLMC is outperformed by the Monte-Carlo algorithm. The plots are shown in \Cref{fig:kh_numerical_stab}.

\begin{figure}[h!]
\includegraphics[width=0.6\textwidth]{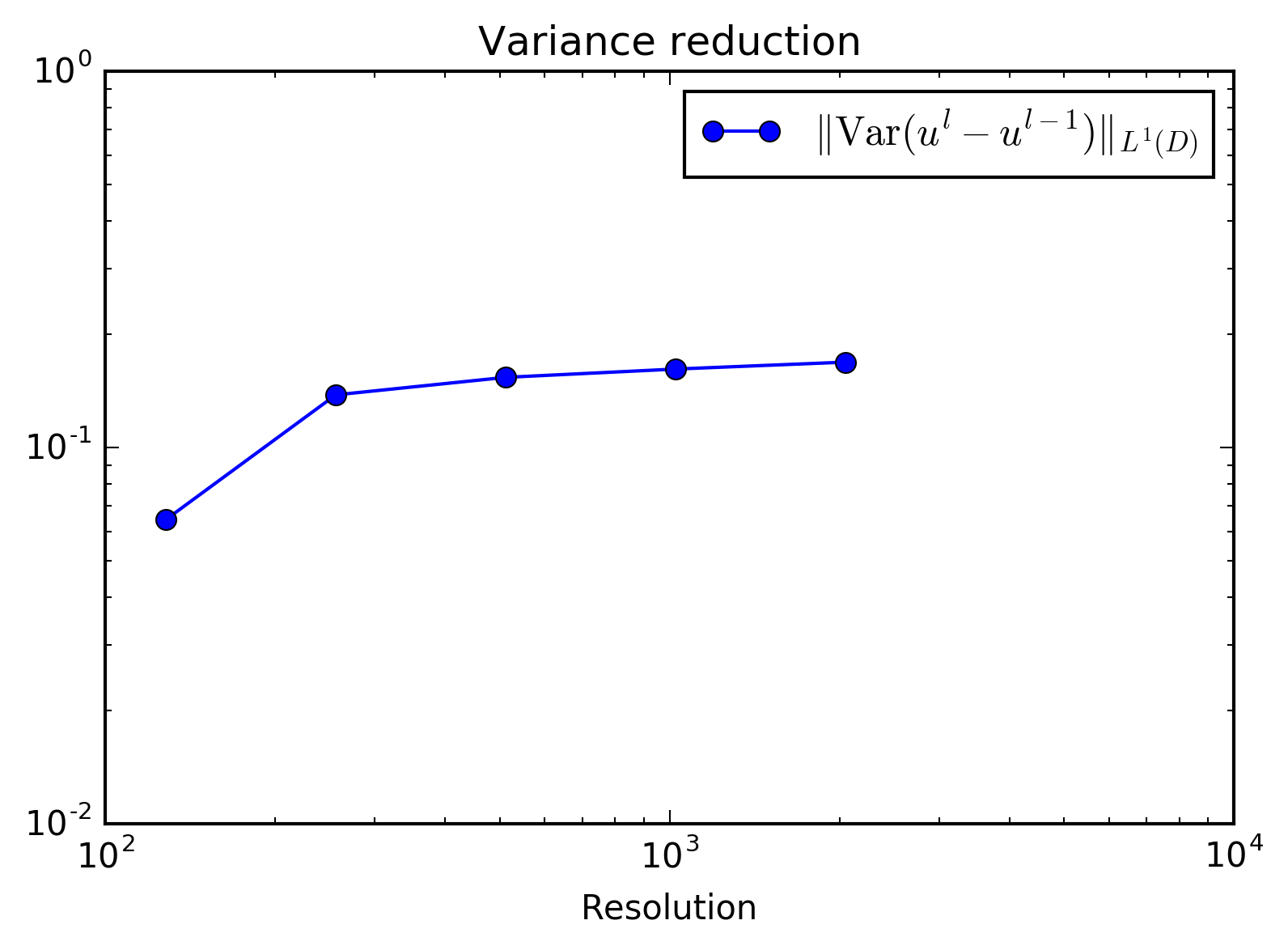}
\caption{Numerically approximated variance between levels,$\Var(\Ypair{\psi}{g(u^l)-g(u^{l-1})})$ for the Kelvin-Helmholtz initial data \eqref{eq:init_kh} at time $T=2$.}
\label{fig:kh_variance_levels}
\end{figure}

\begin{figure}[h!]
\includegraphics[width=0.6\textwidth]{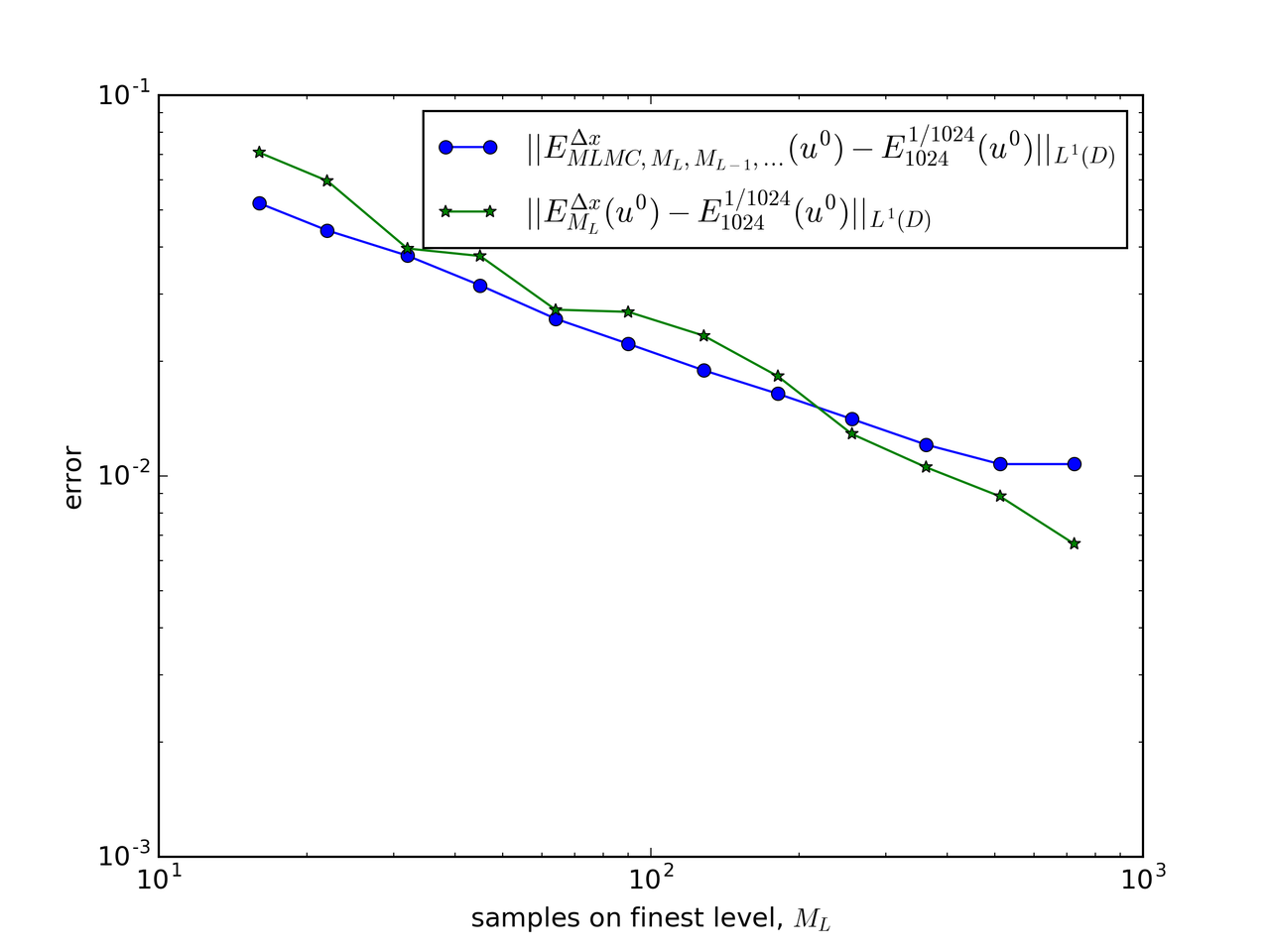}
\caption{Error in computed mean for Kelvin-Helmholtz initial data at time $T=2$.}
\label{fig:kh_mean_mlmc_compare}
\end{figure}

\begin{figure}[h!]
\begin{subfigure}[b]{0.45\textwidth}
\includegraphics[width=\textwidth]{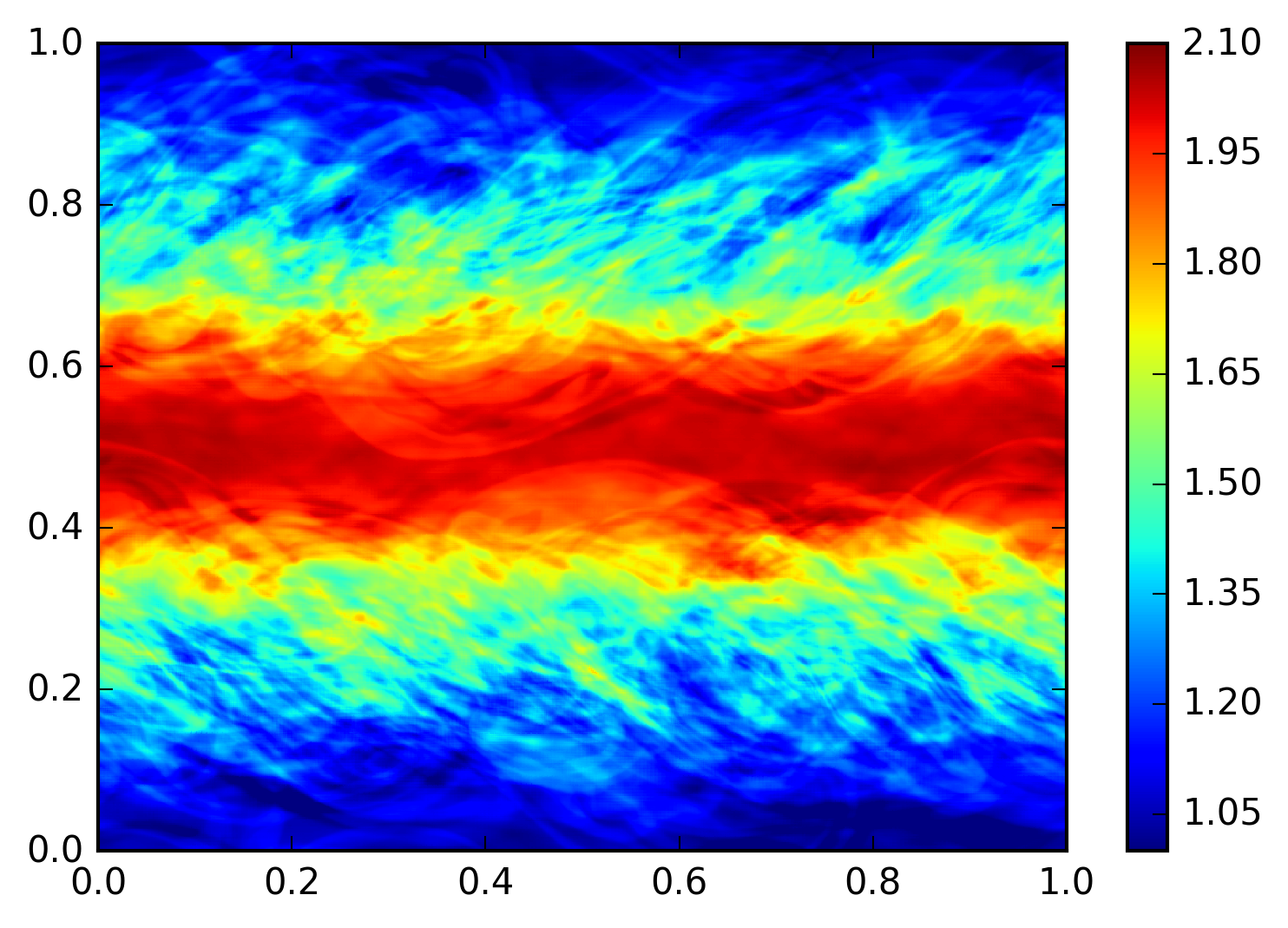}
\caption{Mean from MLMC.}
\end{subfigure}
\begin{subfigure}[b]{0.45\textwidth}
\includegraphics[width=\textwidth]{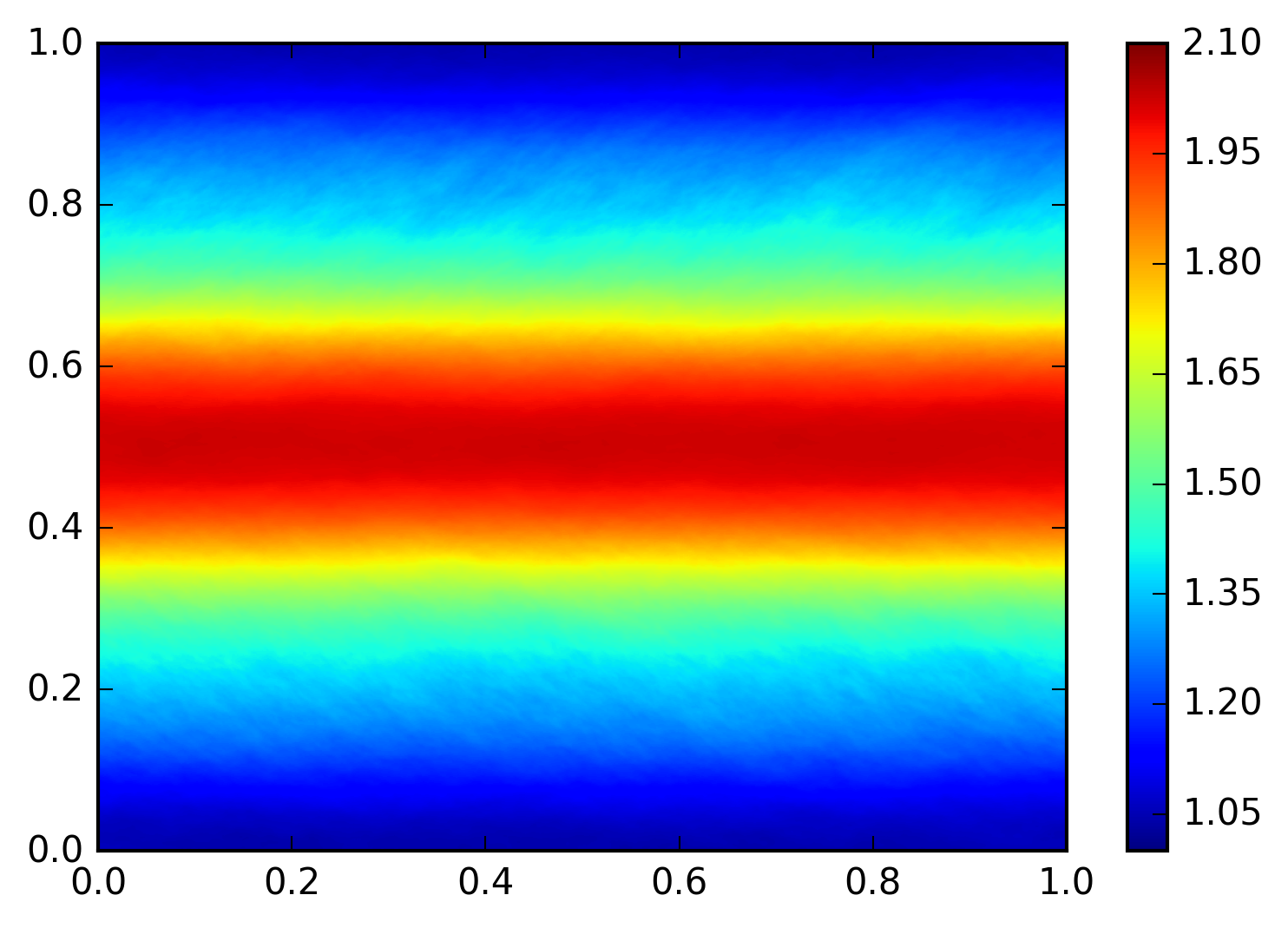}
\caption{Mean from reference solution.}
\end{subfigure}

\begin{subfigure}[b]{0.45\textwidth}
\includegraphics[width=\textwidth]{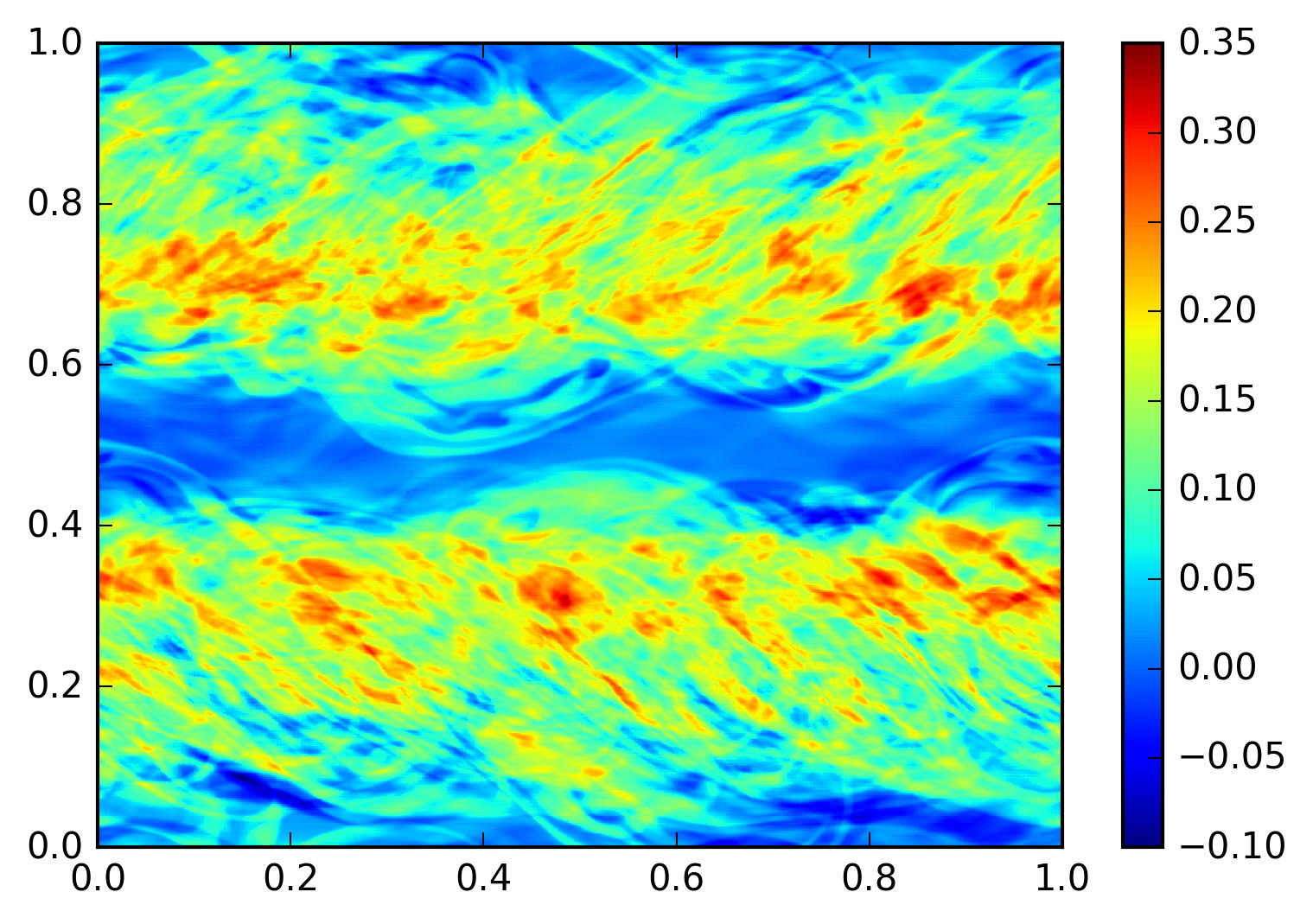}
\caption{Variance from MLMC.}
\end{subfigure}
\begin{subfigure}[b]{0.45\textwidth}
\includegraphics[width=\textwidth]{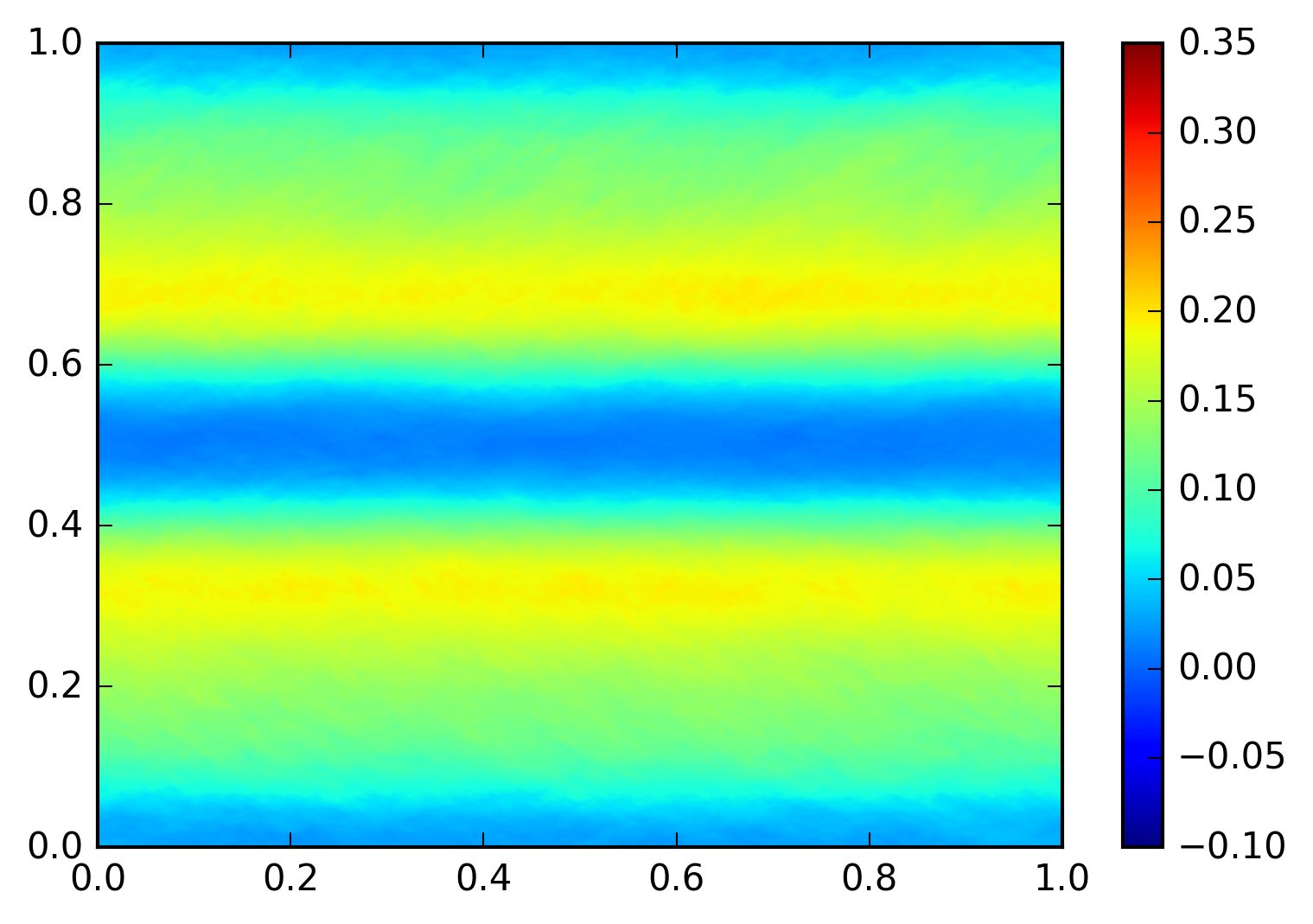}
\caption{Variance from reference solution.}
\end{subfigure}
\caption{Comparison of the singlelevel Monte-Carlo algorithm against the MLMC algorithm for initial data \eqref{eq:init_kh} at time $T=2$. We used $M_L=16$ samples on the finest level for the MLMC computation.}
\label{fig:kh_numerical}
\end{figure}

\begin{figure}[h!]
	\includegraphics[width=0.6\textwidth]{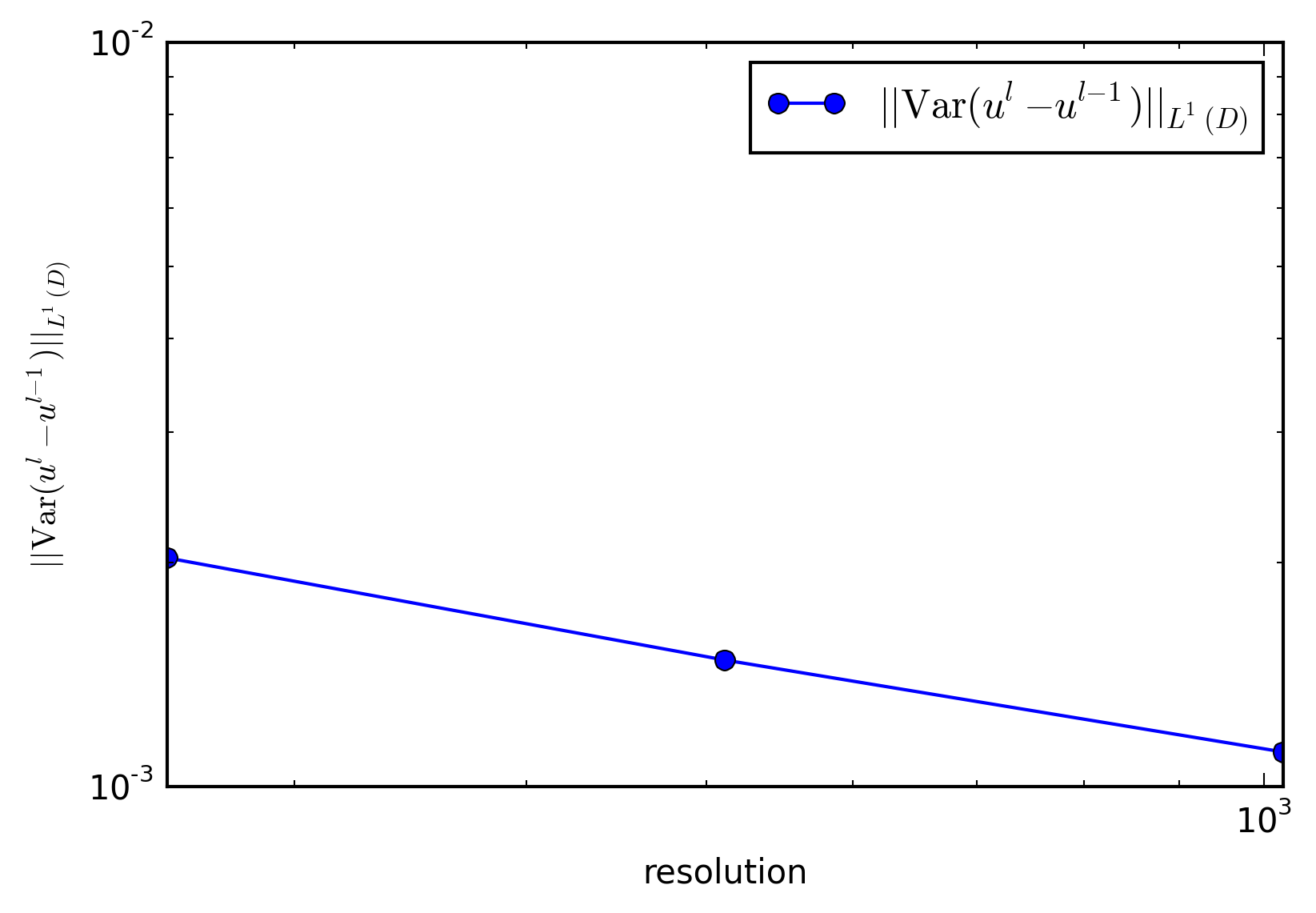}
	\caption{Numerically approximated variance between levels,$\Var(\Ypair{\psi}{g(u^l)-g(u^{l-1})})$ for the Kelvin-Helmholtz initial data \eqref{eq:init_kh} at time $T=2$, using a relaxation time of $T_0=0.05$.}
	\label{fig:mlmc_variance_stabilization}
\end{figure}

\begin{figure}[h!]
	\includegraphics[width=0.6\textwidth]{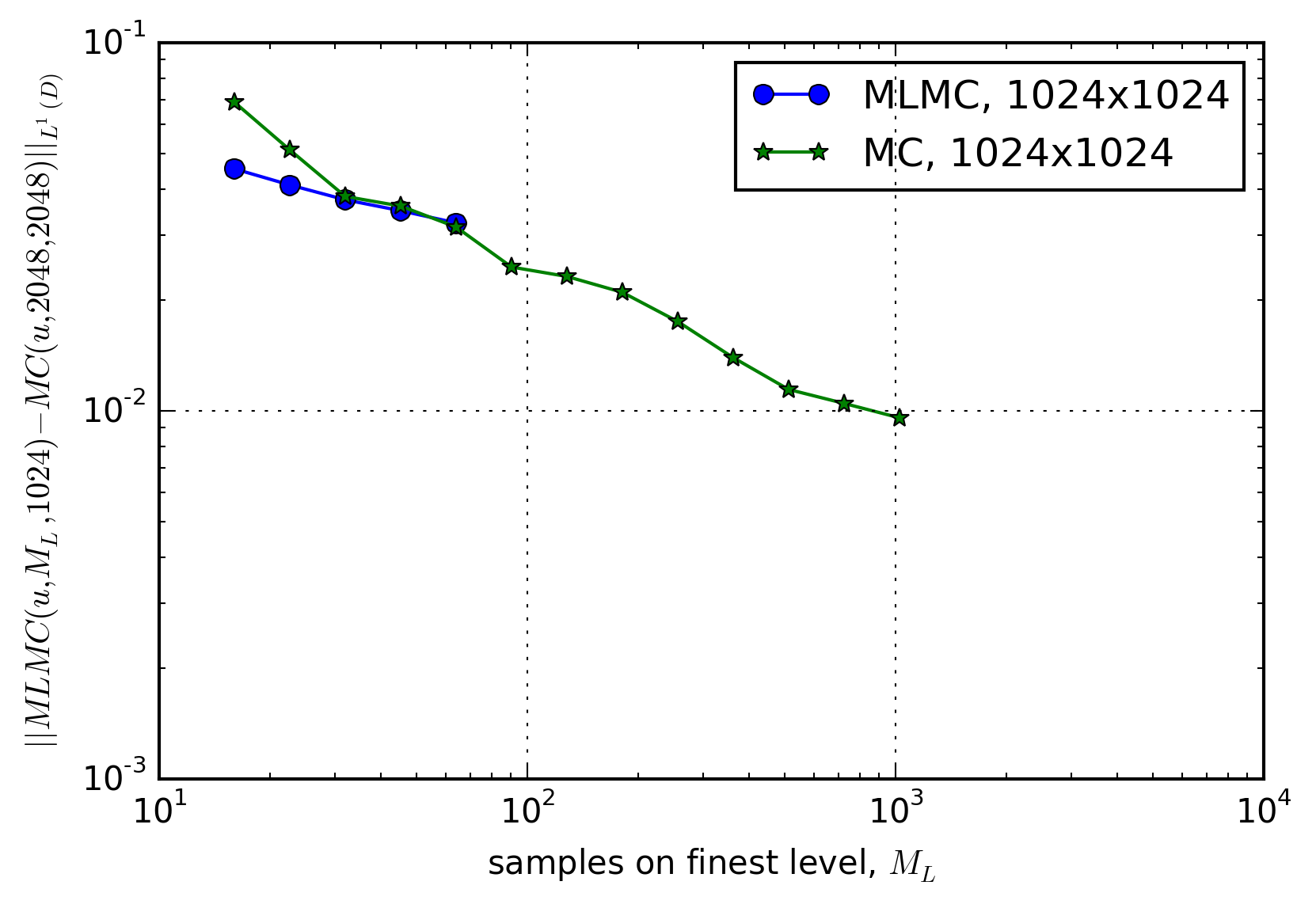}
	\caption{Error in computed mean for Kelvin-Helmholtz initial data at time $T=2$.}
	\label{fig:mlmc_vs_mc_samples_stabilization}
\end{figure}

\begin{figure}[h!]
	\begin{subfigure}[b]{0.45\textwidth}
		\includegraphics[width=\textwidth]{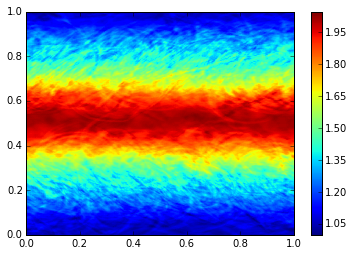}
		\caption{Mean from MLMC with relaxation.}
	\end{subfigure}
	\begin{subfigure}[b]{0.45\textwidth}
		\includegraphics[width=\textwidth]{reference_mean}
		\caption{Mean from reference solution.}
	\end{subfigure}
	
	\begin{subfigure}[b]{0.45\textwidth}
		\includegraphics[width=\textwidth]{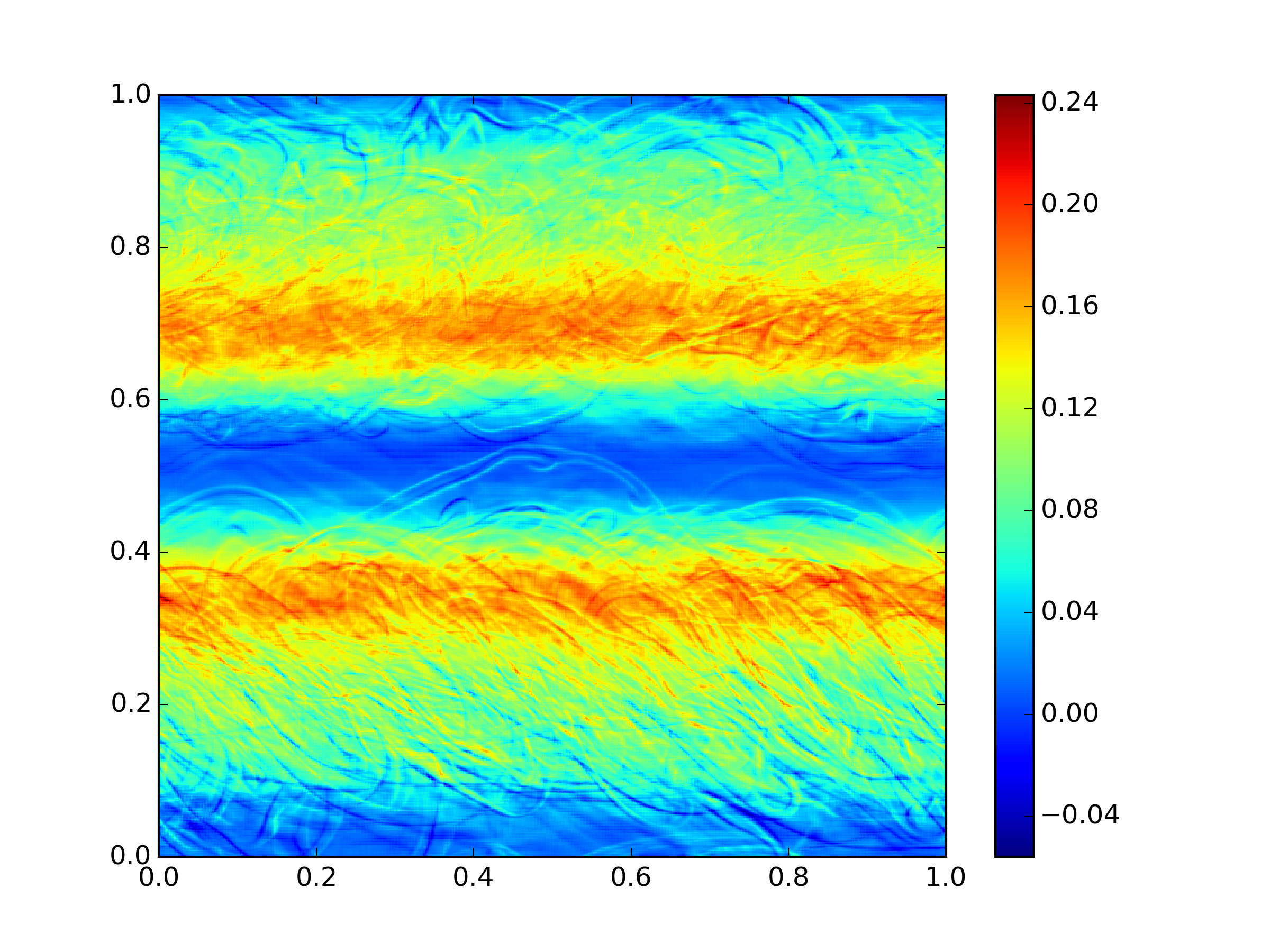}
		\caption{Variance from MLMC with relaxation.}
	\end{subfigure}
	\begin{subfigure}[b]{0.45\textwidth}
		\includegraphics[width=\textwidth]{reference_var}
		\caption{Variance from reference solution.}
	\end{subfigure}
	\caption{Comparison of the singlelevel Monte-Carlo algorithm against the MLMC algorithm for initial data \eqref{eq:init_kh} at time $T=2$. We used $M_L=16$ samples on the finest level for the MLMC computation, we use a relaxation time of $T=0.05$.}
	\label{fig:kh_numerical_stab}
\end{figure}

\clearpage
\section{Conclusions}
In this paper, we reviewed the concept of entropy measure valued solutions for hyperbolic conservation laws. We have laid the theoretical foundations for a multilevel Monte-Carlo algorithm for computing entropy measure valued solutions. 

In \Cref{thm:weak_mlmc}, an error estimate of the MLMC algorithm in the narrow topology was derived. We furthermore derived a precise criterion on the variance decay for gaining an asymptotic speed-up with MLMC compared to singlelevel Monte-Carlo. 

\subsection{Applicability of MLMC for scalar conservation laws}
The theory and numerical experiments reveal, that the MLMC method does work well for approximating EMVS of scalar conservation laws. The numerical experiment agrees with the theory. Furthermore, the MLMC was shown to considerably outperform the MC algorithm both theoretically and through numerical experiments.

\subsection{Applicability of MLMC for systems of conservation laws}
As was made clear by \Cref{thm:number_of_samples}, we can only expect the MLMC algorithm to give a speed-up compared to the MC algorithm if $V_l\to 0$ as $l\to\infty$. The numerical experiments show mixed results in this respect. For the case of the Kelvin-Helmholtz initial data \eqref{eq:init_kh}, the experiments indicated no variance reductions, and the numerical validation agrees. This serves as an example of a case where the measure valued solution is well-defined, and where the Monte-Carlo algorithm converges as a measure, but where the Multilevel Monte-Carlo algorithm can not improve the runtime of singlelevel Monte-Carlo. 

However, in the case of the shockvortex interaction, there is a decay in the variance $V_l$, and as expected, the MLMC algorithm does beat the MC algorithm. 
%
%
\clearpage
\bibliographystyle{siam}
\bibliography{biblo}
\end{document}